\documentclass{amsart}
	\usepackage[latin1]{inputenc}	\usepackage{ae,aecompl,amsbsy,amssymb,amsmath,amsthm,
eurosym,amsfonts,epsfig,graphicx,graphics,verbatim,enumerate,esint,color,MnSymbol}

\usepackage{color,soul}
\usepackage{mathrsfs}

\definecolor{lightblue}{rgb}{.85,.93,1}
\sethlcolor{lightblue}
\newcommand{\norme}[1]{\lvert #1 \rvert_E}
\newcommand{\ch}{\textup{child}}
\newcommand{\cs}{\mathcal{S}}
\newcommand{\cl}{\mathcal{L}}
\newcommand{\ct}{\mathcal{T}}

\newcommand{\ca}{\mathcal{A}}
\newcommand{\cf}{\mathcal{F}}
\newcommand{\cu}{\mathcal{U}}
\newcommand{\cg}{\mathcal{G}}
\newcommand{\bn}{\mathbb{N}}
\newcommand{\br}{\mathbb{R}}
\newcommand{\dmu}{\,\mathrm{d}\mu}
\newcommand{\dnu}{\,\mathrm{d}\nu}
\newcommand{\dx}{\,\mathrm{d}x}
\newcommand{\bz}{\mathbb{Z}}
\newcommand{\be}{\mathbb{E}}
\newcommand{\bp}{\mathbb{P}}
\newcommand{\ba}{\begin{eqnarray*}}
\newcommand{\ea}{\end{eqnarray*}}

\newcommand{\rn}{\mathbb{R}^d}

\newcommand{\es}{E_{\mathcal{S}}}

\newcommand{\normlp}[1]{\lVert#1\rVert_{L^p(E)}}
\newcommand{\normlpp}[1]{\lVert#1\rVert_{L^{p'}(E^*)}}

\newcommand{\lpvv}{L^p(\mathbb{R}^d;E)}
\newcommand{\lpdirect}{L^p(\mathbb{R}^d;E)}
\newcommand{\lpdual}{L^{p'}(\mathbb{R}^d;E^*)}

\newcommand{\vs}{E}
\newcommand{\ve}{e}

\newcommand{\ave}[1]{\langle #1\rangle}

\newcommand{\Exp}[0]{\mathbb{E}}

\def\cyr{\fontencoding{OT2}\fontfamily{wncyr}\selectfont}
\DeclareTextFontCommand{\textcyr}{\cyr}

\newcommand{\abs}[1]{\lvert#1\rvert}

\newcommand{\normes}[1]{\lvert #1 \rvert_{E^*}}

\theoremstyle{plain}
\newtheorem{theorem}{Theorem}[section]
\newtheorem{lemma}[theorem]{Lemma}
\newtheorem{corollary}[theorem]{Corollary}
\newtheorem{proposition}[theorem]{Proposition}

\theoremstyle{definition}
\newtheorem{definition}[theorem]{Definition}

\theoremstyle{definition}

\theoremstyle{remark}

\newtheorem*{remark}{Remark}

\numberwithin{equation}{section}

\newcommand{\norm}[1]{\lVert#1\rVert}

\newcommand{\chs}{\textup{ch}_\mathcal{S}}

\newcommand{\pis}{{\pi_\mathcal{S}}}

\newcommand{\cd}{\mathcal{D}}
\newcommand{\cq}{\mathcal{Q}}

\begin{document}

\date{\today}
     

\title{Operator-valued dyadic shifts and the $T(1)$ theorem}

\author{Timo S. H\"anninen}
\address{Department of Mathematics and Statistics, University of Helsinki, P.O. Box 68, FI-00014 HELSINKI, FINLAND}
\email{timo.s.hanninen@helsinki.fi}
\author{Tuomas P. Hyt\"onen}
\address{Department of Mathematics and Statistics, University of Helsinki, P.O. Box 68, FI-00014 HELSINKI, FINLAND}
\email{tuomas.hytonen@helsinki.fi}

\begin{abstract}
In this paper we extend dyadic shifts and the dyadic representation theorem to an operator-valued setting: We first define operator-valued dyadic shifts and prove that they are bounded. We then extend the dyadic representation theorem, which states that every scalar-valued Calder\'on--Zygmund operator can be represented as a series of dyadic shifts and paraproducts averaged over randomized dyadic systems, to operator-valued Calder\'on--Zygmund operators. As a corollary, we obtain another proof of the operator-valued, global $T1$ theorem. 

We work in the setting of integral operators that have $R$-bounded operator-valued kernels and act on functions taking values in $UMD$-spaces. The domain of the functions is the Euclidean space equipped with the Lebesgue measure. 

In addition, we give new proofs for the following known theorems: Boundedness of the dyadic (operator-valued) paraproduct, a variant of Pythagoras' theorem for (vector-valued) functions adapted to a sparse collection of dyadic cubes, and a decoupling inequality for (UMD-valued) martingale differences.

\end{abstract}
\maketitle
\tableofcontents
\section{Introduction}

In this paper we extend dyadic shifts and the dyadic representation theorem to an operator-valued setting. We work with integral operators that have $R$-bounded operator-valued kernels and act on functions taking values in $UMD$-spaces. The domain of the functions is the Euclidean space equipped with the Lebesgue measure. 

First, we summarize what is known in the scalar-valued setting. A {\it dyadic shift} $S^{ji}$ with parameters $i$ and $j$ (and {\it complexity} $\max\{i,j\}+1$)  is defined by 
$$
S^{ji}f:=\sum_{K\in\cd} D^j_K A_K D^i_Kf,
$$
which involves the following ingredients:
\begin{itemize}
\item the {\it shifted Haar projection} $D_K^i$ associated with a dyadic cube $K\in\cd$ is defined by $$D^i_Kf:= \sum_{\substack{I\in\cd: I\subseteq K,\\\ell(I)={2^{-i}\ell(K)}}}D_If,$$
\item the {\it Haar projection} $D_I$ associated with a dyadic cube $I\in\cd$ is defined by $$D_If:=\sum_{I'\in\textup{child}(I)} \langle f\rangle_{I'}1_{I'}-\langle f \rangle_I 1_I= \sum_{\eta \in \{0,1\}^d\setminus\{0\}}\langle f, h^\eta_I \rangle h^\eta_I,$$ 
where $\textup{child}(I)$ denotes the dyadic children of $I$,  $\langle f \rangle_I:=\frac{1}{\abs{I}}\int_I f\dx$, and $\{h^\eta_I\}_{\eta \in \{0,1\}^d}$ are the Haar functions associated with $I$,
\item the averaging operator $A_K$ is defined by $$A_Kf(x):=\frac{1_K(x)}{\lvert K \rvert}\int_K a_K(x,x')f(x')\,\mathrm{d}x',$$ where it is assumed that the kernels  satisfy  $\lvert a_K(x,x')\rvert \leq 1$ for all $K\in\cd$, $x\in K$, and $x'\in K$.
\end{itemize}
The {\it dyadic paraproduct} associated with a function $b:\br^d\to\br$ is defined by $$\Pi_bf:=\sum_{Q\in\cd} D_Qb\,\langle f \rangle_Q= \sum_{Q\in\cd} \sum_{\eta \in \{0,1\}^d\setminus\{0\}}\langle b, h^\eta_Q \rangle \langle f \rangle_Q h^\eta_Q.$$

Dyadic shifts are bounded on $L^p$ for $1<p<\infty$. Indeed, by Pythagoras' theorem, they are bounded on $L^2$, and, by using the Calder\'on--Zygmund decomposition, from $L^1$ to $L^{1,\infty}$. From the Marcinkiewicz interpolation theorem, it follows that dyadic shifts are bounded on $L^p$ for $1<p\leq 2$, and hence, by duality, on $L^p$ for $2\leq p <\infty$. The weak-$L^1$ bound with an exponential dependence on the complexity was proven by Lacey, Petermichl, and Reguera \cite{lacey2010} and with a linear depence by Hyt\"onen \cite{hytonen2012a}. It is a classical result that a dyadic paraproduct associated with a function $b$ is bounded on $L^p$ if and only if $b$ is a BMO function.

Dyadic shifts are dyadic model operators for Calder\'on--Zygmund operators: Petermichl \cite[Lemma 2.1]{petermichl2000} proved that the Hilbert transform can be represented as a particular dyadic shift averaged over randomized dyadic systems, and  Hyt\"onen \cite[Theorem 4.2]{hytonen2012a} that every Calder\'on--Zygmund operator can be represented as a series of dyadic shifts and paraproducts averaged over randomized dyadic systems. The dyadic representation theorem for Calder\'on--Zygmund operators together with the boundedness of dyadic shifts and paraproducts yields another proof of the global $T1$ theorem for Calder\'on--Zygmund operators. For a detailed proof of the dyadic representation theorem, see the lecture notes on the $A_2$ theorem \cite{hytonen2011}.

The operator-valued setting in this paper follows the by-now-usual paradigm of doing Banach-space valued harmonic analysis beyond Hilbert space: Orthogonality of vectors is replaced with unconditionality of martingale differences, and uniform boundedness of operators with $R$-boundedness. Pioneering examples of this are the result by Burkholder \cite{burkholder1981} and Bourgain \cite{bourgain1983} that the Hilbert transform is bounded on $L^p(E)$ if and only if the Banach space $E$ has the UMD property, and the operator-valued Fourier multiplier theorems by Weis \cite{weis2001}. 

A family of operators $\ct\subseteq \cl(E,F)$ from a Banach space $(E,\lvert \,\cdot\, \rvert_E)$ to a Banach space $(F,\lvert\,\cdot\, \rvert_F)$ is said to be {\it $R$-bounded} if there exists a constant $\mathcal{R}_p(\ct)$ such that
$$
\Big(\be \lvert \sum_{n=1}^N \varepsilon_n T_ne_n \rvert_F ^p \Big)^{1/p} \leq \mathcal{R}_p(\ct) \Big(\be \lvert \sum_{n=1}^N \varepsilon_ne_n \rvert_E ^p \Big)^{1/p} 
$$
for all choices of operators $(T_n)_{n=1}^N\subseteq \ct$ and vectors $(e_n)_{n=1}^N\subseteq E$, where the expectation is taken over independent, unbiased random signs $(\varepsilon_n)_{n=1}^N$. A Banach space $(E,\lvert\,\cdot \,\rvert_E)$ is said to be a {\it UMD (unconditional martingale difference) space} if there exists a constant $\beta_p(E)$ such that
$$
\norm{\sum_{n=1}^N \epsilon_n d_n}_{L^p(E)}\leq \beta_p(E) \norm{\sum_{n=1}^N d_n}_{L^p(E)}
$$
for all $E$-valued $L^p$-martingale difference sequences $(d_n)_{n=1}^N$ and for all choices of signs $(\epsilon_n)_{n=1}^N \in \{-1,+1\}^N$. It is well-known that $R$-boundedness and UMD-property are independent (up to the involved constants) of the exponent $p\in(1,\infty)$; For an exposition on Banach-space-valued martingales, UMD spaces, and $R$-boundedness,  among other things, see Neerven's lecture notes \cite{neerven2008}. 

We conclude the introduction by  precisely fixing the operator-valued setting and stating the results. First, we define the operator-valued dyadic shifts and state their boundedness.

\begin{definition}[Operator-valued dyadic shift]Let $E$ be a UMD space. An {\it operator-valued dyadic shift} associated with parameter $j$ and $i$ is defined by
$$
S^{ji}f:=\sum_{K}D^j_K A_K D^i_Kf$$
for every locally integrable function $f:\br^d\to E$, where, for each $K\in\cd$, the averaging operator $A_K$ associated with an operator-valued kernel $a_K:\br^d\times \br^d\to \cl(E)$ is defined by $$A_Kf(x):=\frac{1_K(x)}{\lvert K \rvert}\int_K a_K(x,x')f(x')\,\mathrm{d}x.$$ The family of the operator-valued kernels is assumed to be $R$-bounded so that there exists a positive constant $\mathcal{R}_p(\{a\})$ such that $$\mathcal{R}_p(\{a_K(x,x')\in\cl(E) : K\in\cd, x\in K, x'\in K\})\leq \mathcal{R}_p(\{a\}).$$ 

\end{definition}
Let $L^p(\br^d;E)$ denote the Lebesgue--Bochner space, which is equipped with the norm
$$
\norm{f}_{L^p(\br^d;E)}=\Big(\int_{\br^d} \lvert f(x) \rvert_E^p \,\mathrm{d}x \Big)^{1/p}.
$$
We prove the following theorem:
\begin{theorem}[Operator-valued dyadic shifts are bounded]\label{dyadicshiftsarebounded}Let $1<p<\infty$. Let $\vs$ be a UMD space. Let $S^{ij}$ be a {\it dyadic shift with parameters} $i$ and $j$ and associated with the operator-valued kernels $a_K$. Then
$$
\lVert S^{ij}f \rVert_{\lpvv}\leq 4  (\max\{i,j\}+1) \,\mathcal{R}_p(\{a\}) \beta_p(E)^2 \lVert f \rVert_{\lpvv}
$$
for all $f\in \lpvv$.
\end{theorem}

Next, we define the operator-valued Calder\'on--Zygmund operators and state the dyadic representation theorem for them. Following the paradigm of replacing orthogonality by unconditionality of martingale differences and uniform boundedness by $R$-boundedness,
the standard estimates and the weak boundedness property are replaced by the {\it  Rademacher standard estimates} and {\it Rademacher weak boundedness property}.

\begin{definition}[Rademacher standard estimates] An operator-valued singular kernel $k:\rn\times\rn\setminus \{(x,x):x\in \rn\} \to \cl(E)$ satisfies the {\it Rademacher standard estimates} if and only if:  
\begin{enumerate}
\item[(i)]The kernel $k$ satisfies the decay estimate $$\mathcal{R}(\{k(x,y)\lvert x-y \rvert^{d}  : x\in\mathbb{R}^d, y\in\mathbb{R}^d \text{ with } x\neq y\})\leq \mathcal{R}_{\text{CZ}_0}$$
for some constant $\mathcal{R}_{\text{CZ}_0}$.
\end{enumerate}

\begin{enumerate}
\item[(ii)]The kernel $k$ satisfies the H\"older-type estimates  \begin{equation*}
\begin{split}\mathcal{R}\Bigg(\Bigg\{&(k(x,y)-k(x',y))\left(\frac{ \lvert x-y \rvert}{\lvert x-x' \rvert}\right)^{\alpha} \lvert x-y \rvert^{d} : \\
&x\in\mathbb{R}^d, x'\in\mathbb{R}^d, y\in\mathbb{R}^d \text{ with } 0<\lvert x-x' \rvert<\frac{1}{2}\lvert x-y \rvert)\Bigg\}\Bigg)\leq \mathcal{R}_{\text{CZ}_\alpha}
\end{split}
\end{equation*}
and
\begin{equation*}
\begin{split}\mathcal{R}\Bigg(\Bigg\{&(k(x,y)-k(x,y'))\left(\frac{ \lvert x-y \rvert}{\lvert y-y' \rvert}\right)^{\alpha} \lvert x-y \rvert^{d} : \\
&x\in\mathbb{R}^d, y\in\mathbb{R}^d, y'\in\mathbb{R}^d \text{ with } 0<\lvert y-y' \rvert<\frac{1}{2}\lvert x-y \rvert)\Bigg\}\Bigg)\leq \mathcal{R}_{\text{CZ}_\alpha}
\end{split}
\end{equation*}
for some H\"older exponent $\alpha \in(0,1]$ and for some constant $\mathcal{R}_{\text{CZ}_\alpha}$.
\end{enumerate}
\end{definition}

\begin{definition}[Rademacher weak boundedness property]An operator $T$ mapping locally integrable $E$-valued functions to locally integrable $E$-valued functions satisfies the {\it Rademacher weak boundedness property} if and only if 
$$
\mathcal{R}\Big(\Big\{ \frac{1}{\lvert I \rvert} \int_{\br^d} 1_I(x) T(\,\cdot\, 1_I )(x)\, \mathrm{d}x\in \cl(E): I\in\mathcal{D} \Big\}\Big)\leq \mathcal{R}_{\text{WBP}}
$$
for some constant $\mathcal{R}_{\text{WBP}}$.
\end{definition}

The randomized dyadic systems are defined as follows. Let $\mathcal{D}^0$ designate the standard dyadic system. For every parameter $(\omega_j)_{j\in\mathbb{Z}}\in (\{0,1\}^d)^{\mathbb{Z}}=:\Omega$ and every $I\in\mathcal{D}^0$, the translated dyadic cube $I\dot+ \omega$ is defined by $$
I\dot+ \omega:= I+\sum_{j: 2^{-j} < \ell(I)}2^{-j}\omega_j.
$$
For each $\omega\in\Omega$, the translated dyadic system $\mathcal{D}^\omega$ is defined by $\mathcal{D}^\omega:=\{I\dot+\omega : I\in\mathcal{D}^0\}. $ We equip the parameter set with the natural probability measure: Each component $\omega_j\in \{0,1\}^d$ has an equal probability $2^{-d}$ of taking any of the $2^d$ values and all components are stochastically independent. 

\begin{theorem}[Operator-valued dyadic representation theorem]\label{dyadicrepresentation}Let $\vs$ be a Banach space. Let $T$ be a singular integral operator that satisfies the Rademacher weak boundedness property and whose operator-valued kernel satisfies the Rademacher standard estimates with the H\"older exponent $\alpha$. Assume that $T:\lpvv\to\lpvv$ is bounded. Then, for some dyadic shifts $S^{ij}_{\mathcal{D}^\omega}$ and for the dyadic paraproducts $\Pi^{\mathcal{D}^\omega}_{T1}$ and $\Pi^{\mathcal{D}^\omega}_{T^*1}$, we have
\begin{equation*}
\begin{split}
\langle g,Tf\rangle=\,\mathbb{E}_\omega \Big(&C_T\sum_{i\geq 0, j\geq0}2^{(1/\epsilon)} 2^{-(1-\epsilon)\alpha\max\{i,j\}}\langle g,S^{ij}_{\mathcal{D}^\omega}f\rangle\\
& +\langle g,(\Pi^{\mathcal{D}^\omega}_{T1}+(\Pi^{\mathcal{D}^\omega}_{T^*1})^*)f\rangle \Big)
\end{split}
\end{equation*}
for all $g\in C^1_0(\mathbb{R}^d;\mathbb{R})\otimes E^*$ and $f\in C^1_0(\mathbb{R}^d;\mathbb{R})\otimes E$. Moreover, 
$$
C_T\lesssim_{d,\alpha} \mathcal{R}_{CZ_0} +\mathcal{R}_{CZ_\alpha}+\mathcal{R}_\text{WBP}.
$$
\end{theorem}
\begin{remark}The statement contains an auxiliary parameter $\epsilon$ with $0<\epsilon<1$. The factor $2^{(1/\epsilon)} 2^{-(1-\epsilon)\alpha\max\{i,j\}}$ can be replaced with the factor $(1+\max\{i,j\})^{\gamma(d+\alpha)}2^{-\alpha\max\{i,j\}}$. This is achieved by replacing the \lq boundary\rq\, function $t\mapsto t^\gamma$ with the function $t\mapsto (1+a^{-1}\log (t^{-1}))^{-\gamma}$ in the definition of a good dyadic cube, Definition \ref{def_goodcube}, which then results in the decay $2^{-\alpha\max\{i,j\}}$ in the estimates for the matrix elements, Lemma \ref{firstcase} and Lemma \ref{secondcase}. For the details, see the lecture notes on the $A_2$ theorem \cite{hytonen2011}. For simplicity, we use the function $t\mapsto t^\gamma$.
\end{remark}

 For a Banach space $(\ct,\lvert\, \cdot \,\rvert_\ct)$, the $\text{BMO}_p(\br^d;\ct)$-norm is defined by
$$
\norm{b}_{\text{BMO}_p(\br^d;\ct)}:=\sup_{Q\in\cd}\Big(\frac{1}{\lvert Q \rvert} \int_Q \lvert b(x)-\langle b \rangle_Q \rvert_\ct^p\,\mathrm{d}x\Big)^{1/p}.
$$The following sufficient condition for the boundedness of the paraproduct $\Pi_b$ associated with an operator-valued function $b$ was proven by Hyt\"onen \cite{hytonen2008} by using interpolation and decoupling of martingale differences. Predecessors of this operator-valued result (under stronger assumptions) were obtained by Hyt\"onen and Weis \cite{hytonen2006a}, based on unpublished ideas of Bourgain recorded by Figiel and Wojtaszczyk \cite{figiel2001} in the case of a scalar-valued function $b$.
\begin{theorem}[Sufficient conditions for the boundedness of a paraproduct]\label{sufficientcondition}Let $\vs$ be a UMD space. Let $\ct \subseteq \cl(\vs)$ be a UMD subspace of $\cl(\vs)$. Then
$$
\norm{\Pi_bf}_{L^p(\br^d;E)}\leq 6\cdot 2^{d} pp'\beta_p(E)^2 \beta_p(\ct) \norm{b}_{\text{BMO}_p(\br^d;\ct)} \norm{f}_{L^p(\br^d;E)}
$$
for all $b\in \text{BMO}_p(\br^d;\ct)$ and $f\in \lpvv$.
\end{theorem}
In this paper we give a different proof of Theorem \ref{sufficientcondition}. This proof is elementary in that neither interpolation nor decoupling of martingale differences is used.

By combining Theorem \ref{dyadicshiftsarebounded}, Theorem \ref{dyadicrepresentation}, and Theorem \ref{sufficientcondition},
we obtain a new proof for the following corollary, which is a special case of Hyt\"onen's vector-valued, non-homogeneous, global $Tb$ theorem \cite[Tb theorem 4]{hytonen2008}. Earlier results of this type include the first vector-valued $T1$ theorem by Figiel \cite{figiel1990a}, and the first operator-valued $T1$ theorem by Hyt\"onen and Weis \cite{hytonen2006a}. Several related results have appeared in the literature.
\begin{corollary}[$T1$ theorem for operator-valued kernels] Let $T$ be a singular integral operator that satisfies the Rademacher weak boundedness property and whose operator-valued kernel satisfies the Rademacher standard estimates. Assume that $T1\in BMO_p(\br^d;\ct)$ and $T^*1 \in BMO_p(\br^d;\ct^*)$ for some UMD subspaces $\ct \subset \cl(E)$ and $\ct^*\subset \cl(E^*)$. Then
\begin{equation*}
\begin{split}
&\norm{T}_{L^p(\br^d;E)\to L^p(\br^d;E)} \\
&\lesssim_{\ct,d,p,\alpha} (\mathcal{R}_{CZ_0} +\mathcal{R}_{CZ_\alpha}+\mathcal{R}_\text{WBP}+\norm{T1}_{BMO_p(\br^d;\ct)}+\norm{T^*1}_{BMO_p(\br^d;\ct^*)})\beta_p(E)^2 .
\end{split}
\end{equation*}
\end{corollary}
Here the condition $T^*1 \in BMO_p(\br^d;\ct^*)$ is interpreted via duality as follows: There exists $b\in BMO_p(\br^d;\ct^*)$ such that $\left(\int_{\br^d} \, T(\,\cdot\, h_I)(x) \mathrm{d}x \right)^*=\int_{\br^d} b(x) h_I(x) \mathrm{d}x$. This interpretation originates from extracting the paraproducts as in the equation \eqref{t1interpretation} in Section \ref{section_representation}.

Finally, let us compare our results with Pott and Stoica's results \cite{pott2014}. They study the question how the operator norm of a general vector-valued Calder\'on--Zygmund operator depends on the UMD constant. The purpose of their paper is to prove that this dependence is linear for a large class of Calder\'on--Zygmund operators. They prove the following estimate for vector-valued dyadic shifts:
\begin{theorem}[Self-adjoint vector-valued dyadic shifts depend linearly on the UMD constant \cite{pott2014}]\label{dyadicshiftsareboundedpott}Let $1<p<\infty$. Let $\vs$ be a UMD space. Let $S^{ij}$ be a self-adjoint dyadic shift with parameters $i$ and $j$. Then
$$
\lVert S^{ij}f \rVert_{L^p(\br;E)\to L^p(\br;E)}\lesssim (\max\{i,j\}+1)2^{\max\{i,j\}/2} \beta_p(E) \lVert f \rVert_{\lpvv}
$$
for all $f\in \lpvv$.
\end{theorem}
By the fact that an estimate for dyadic shifts can be transferred to an estimate for Calder\'on--Zygmund operators via the dyadic representation theorem (Theorem \ref{dyadicrepresentation}), their estimate for dyadic shifts then transfers to the following estimate for vector-valued Calder\'on--Zygmund operators:
\begin{theorem}[Calder\'on--Zygmund operators that have even kernel with sufficiently smoothness,  and vanishing paraproduct depend linearly on the UMD constant \cite{pott2014}]\label{linearumdpott}Let $1<p<\infty$. Let $E$ be a UMD space. Let $T$ be a singular integral operator that satisfies the weak boundedness property and whose kernel satisfies the standard estimates with the H\"older-exponent $\alpha$. Assume that the kernel is even and has smoothness $\alpha>1/2$. Assume that $T$ satisfies the vanishing paraproduct condition $T(1)=T^*(1)=0$. Then
$$
 \norm{T}_{L^p(\br;E)\to L^p(\br;E)} \lesssim_{\alpha,d}  C_T \beta_p(E),
$$ 
where $C_T$ depends only on the constants in the standard estimates and the weak boundedness property.
\end{theorem}
Now, let us compare our estimate for dyadic shifts (Theorem \ref{dyadicshiftsarebounded}) with Pott and Stoica's estimate (Theorem \ref{dyadicshiftsareboundedpott}). We note that the dependence on the complexity dictates whether the series in the dyadic representation theorem (Theorem \ref{dyadicrepresentation}) converges. On the one hand, our estimate depends linearly on the complexity, whereas theirs exponentially, which then translates into the smoothness condition $\alpha>1/2$ in their estimate for Calder\'on--Zygmund operators (Theorem \ref{linearumdpott}). On the other hand, their estimate depends linearly on the UMD constant, whereas ours depends quadratically. We remark that by interpolating between our estimate and theirs (by multiplying the inequalities $\norm{S}_{L^p(\br;E)\to L^p(\br;E)}^{1-\theta} \lesssim k^{1-\theta} \beta_p(E)^{2(1-\theta)}$ and $\norm{S}_{L^p(\br;E)\to L^p(\br;E)}^{\theta}\lesssim 2^{\theta k/2}\beta_p(E)^\theta$), we obtain that
$$
\norm{S}_{L^p(\br;E)\to L^p(\br;E)}\lesssim \beta_p(E)^{2-\theta} k^{1-\theta}2^{\theta k/2},
$$
which then transfers to:
\begin{corollary}[Calder\'on--Zygmund operators that have even kernel and vanishing paraproduct depend subquadratically on the UMD constant]Let $1<p<\infty$. Let $E$ be a UMD space. Let $T$ be a singular integral operator that satisfies the weak boundedness property and whose kernel satisfies the standard estimates with the H\"older-exponent $\alpha$. Assume that the kernel is even. Assume that $T$ satisfies the vanishing paraproduct condition $T(1)=T^*(1)=0$. Then
$$
 \norm{T}_{L^p(\br;E)\to L^p(\br;E)} \lesssim_{\alpha,d}  C_T \left\{
     \begin{array}{lr}
      \frac{1}{(\alpha-\theta)^c}\beta_p(E)^{2(1-\theta)}  & \text{ for } \alpha\leq 1/2,\\
       \beta_p(E) & \text{ for } \alpha> 1/2,
     \end{array}
   \right.
$$ 
for every $\theta$ with $0<\theta<\alpha$. Here $C_T$ depends only on the constants in the standard estimates and the weak boundedness property.
\end{corollary}
Lastly, we remark that we prove our estimate for dyadic shifts by using a martingale decoupling equality, whereas Pott and Stoica prove theirs by using the Bellman function method. At the moment, we do not know how to reproduce their result by our method nor our result by their method. A more complete understanding of both methods could yield interesting further results.

\section{Preliminaries}\label{section_tools}

\subsection{Sum of stochastically independent conditional expectations}
\begin{lemma}[Sum of stochastically independent conditional expectations]\label{independentfunctions}Let $(X_n,\cf_n,\mu_n)$ be a probability space for each $n=1,\ldots,N$. Let $(X,\cf,\mu)$ denote the product probability space $(\Pi_{n=1}^N X_n,\bigtimes_{n=1}^N \cf_n,\bigtimes_{n=1}^N \mu_n)$. Let $1\leq p\leq \infty$. Assume that $f_n\in L^p(X_n,\cf_n,\mu_n;E)$ and that $\cg_n$ is a sub-$\sigma$-algebra of $\cf_n$ for each $n=1,\ldots,N$.   Then
$$
\norm{\sum_{n=1}^N \be [f_n \lvert \cg_n]}_{L^p(X,\cf,\mu;E)}\leq  \norm{\sum_{n=1}^N f_n }_{L^p(X,\cf,\mu;E)}.
$$
\end{lemma}

\begin{proof} 
We prove that $\be[f_n \lvert \cg_n]=\be[f_n \lvert \bigtimes_{m=1}^N\cg_n]$, from which the estimate follows by the linearity and the $L^p$-contractivity of the conditional expectation operator,
$$
\norm{\sum_{n=1}^N \be [f_n \lvert \cg_n]}_{L^p(X,\cf,\mu;E)}=\norm{\be [\sum_{n=1}^N f_n \lvert \bigtimes_{m=1}^N\cg_n]}_{L^p(X,\cf,\mu;E)} \leq  \norm{\sum_{n=1}^N f_n }_{L^p(X,\cf,\mu;E)}.
$$ 
By Kolmogorov's definition of the conditional expectation, we have $\be[f_n \lvert \cg_n]=\be[f_n \lvert \bigtimes_{m=1}^N\cg_n]$ if and only if
\begin{equation}\label{kolmogorovdefinition}
\int_G \be[f_n \lvert \cg_n] \mathrm{d}\mu=\int_G f_n \mathrm{d}\mu
\end{equation}
for all $G\in \bigtimes_{m=1}^N\cg_n$. The collection of sets $G\in \bigtimes_{m=1}^N\cg_n$ satisfying the condition \eqref{kolmogorovdefinition} is a $\lambda$-system (which means that the collection contains the empty set, is closed under taking complements and is closed under taking countable disjoint unions). The $\sigma$-algebra $ \bigtimes_{m=1}^N\cg_n$ is generated by the collection of sets $G_1\times\cdots\times G_N$ with each $G_n\in \cg_n$, which is a $\pi$-system (which means that the collection is closed under taking finite intersections). Dynkin's $\pi-\lambda$ theorem (for a proof, see, for example, the appendix of Durrett's textbook \cite{durrett2010}) states that the $\lambda$-system and the $\sigma$-algebra both generated by the same $\pi$-system coalesce. Hence it suffices to check the condition \eqref{kolmogorovdefinition} for the sets $G_1\times\cdots\times G_N$ with each $G_n\in \cg_n$, which is done by using Fubini's theorem and Kolmogorov's definition of the conditional expectation, 
\begin{equation*}
\begin{split}
\int_{G_1\times\cdots\times G_N} \be[f_n \lvert \cg_n] \mathrm{d}\mu &= \int_{G_n} \be[f_n \lvert \cg_n] \mathrm{d}\mu_n \prod_{m\neq n} \dmu_m(G_m) \\
&= \int_{G_n} f_n  \mathrm{d}\mu_n \prod_{m\neq n} \dmu_m(G_m)= \int_{G_1\times\cdots\times G_N} f_n  \mathrm{d}\mu.
\end{split}
\end{equation*}

\end{proof}
\subsection{Properties of $R$-bounds}In this section we have collected some properties of $R$-bounds. For the proofs and references, see Neerven's lecture notes  \cite{neerven2008}.
 Let $(X,\cf,\mu)$ be a $\sigma$-finite measure space. Let $\vs$ be a Banach space. Assume that $x\mapsto L(x)$ is an $\cl(E)$-valued function defined on $X$ such that the function $x\mapsto L(x)\ve$ defined on $X$ is strongly measurable for each $\ve\in\vs$. We define the operator $\int_{X} L(x)\lambda(x)\mathrm{d}\mu(x):\vs\to\vs$ by
$$
(\int_{X} L(x)\lambda(x)\mathrm{d}\mu(x))\ve:=\int_{X} L(x)\ve\lambda(x)\mathrm{d}\mu(x) \quad\text{ for all } \ve\in\vs.
$$

\begin{proposition}[Averaging preserves $R$-bounds]\label{propertyintegral} Let $(X,\cf,\mu)$ be a $\sigma$-finite measure space. Let $S$ be an index set.
Let $\{L_s\}_{s\in S}$ be an indexed family of $\cl(E)$-valued functions defined on $X$ such that the $E$-valued function $x\mapsto L_s(x)e$ defined on $X$ is strongly $\mu$-measurable for every $e\in E$ and every $s\in S$.  Let $\{\lambda_s\}_{s\in S}$ be an indexed family of integrable real-valued functions. Then
\begin{equation*}
\begin{split}
&\mathcal{R}(\{\int_{X} L_s(x)\lambda_s(x)\mathrm{d}\mu(x)  : s\in S\})\\
&\leq \sup\{ \int_X \lvert \lambda_s(x) \rvert \mathrm{d}\mu(x) : s\in S\} \cdot\mathcal{R}(\{L_s(x) : s\in S \text{ and } x\in X\}).
\end{split}
\end{equation*}
\end{proposition}
\begin{proposition}[Triangle inequality for $R$-bounds]\label{propertytriangle}Let $S$ and $T$ be index sets. Let $\{L_s\}_{s\in S}$ and $\{M_t\}_{t\in T}$ be indexed families of operators. Then
$$
\mathcal{R}(\{M_s+L_t : s\in S, t\in T\})\leq \mathcal{R}(\{M_s: s\in S\})+\mathcal{R}(\{L_t : t\in T\}).
$$
\end{proposition}

\begin{proposition}[Vector-valued Stein's inequality]\label{steininequality}Let $E$ be a UMD space. Let $(\Omega,\cf,\bp)$ be a probability space. Let $(\cg_n)_{n=1}^\infty$ be a refining sequence of $\sigma$-algebras. Then the family
$$
\{\be[\,\cdot\, \lvert \cg_n]:L^p(\Omega;E)\to L^p(\Omega;E) \}_{n=1}^\infty
$$
is R-bounded. Moreover,
$$
\mathcal{R}_p\big(\{\be[\,\cdot\, \lvert \cg_n]:L^p(\Omega;E)\to L^p(\Omega;E)  \}_{n=1}^\infty\big)\leq \beta_p(E).
$$
\end{proposition}

\subsection{Pythagoras' theorem for functions adapted to a sparse collection}

Let $\mu$ be a Borel measure on $\br^d$. We use the notation $\langle f \rangle_Q^\mu:=\frac{1}{\mu(Q)}f \dmu.$ Let $\cs$ be a collection of dyadic cubes. For each $S\in\cs$, let $\chs(S)$ denote the collection of all maximal $S'\in\cs$ such that $S'\subsetneq S$ and let $\es(S)$ denote the set $\es(S):=S\setminus \bigcup_{S'\in\chs(S)} S'$. For each $Q\in\cd$, let $\pis(Q)$ denote the minimal dyadic cube $S\in\cs$ such that $S\supseteq Q$. We say that the collection $\cs$ is {\it sparse} if $\mu( \es(S)) \geq \frac{1}{2} \mu(S)$ for every $S\in\cs$.

\begin{lemma}[Special case of the dyadic Carleson embedding theorem]\label{carlesonembedding}Let $E$ be a Banach space. Let $1<p<\infty$. Assume that $\cs$ is a sparse collection. Then
$$
(\sum_{S\in\cs}(\langle \norme{f} \rangle_{S}^\mu)^p \mu(S))^{1/p}\leq 2p'\norm{f}_{L^p(\mu;E)}.
$$
\end{lemma}
\begin{proof}
For the dyadic Hardy--Littlewood maximal function $M^\mu f=\sup_{Q\in\cd} 1_Q\langle f \rangle^\mu_Q$, we have $\langle \norme{f}\rangle_{S}^\mu \leq \inf_{S}M^\mu \norme{f}$, and, moreover, $\norm{M \norme{f}}_{L^p(\mu)}\leq p' \norm{\norme{f}}_{L^p(\mu)}$. These facts together with the assumptions
yield 
\begin{equation*}
\begin{split}
&\big(\sum_{S\in\cs}(\langle \norme{f} \rangle^\mu_S)^p\mu(S)\big)^{1/p}\leq 2^{1/p}\big(\sum_{S\in\cs}\int_{\es(S)}(\inf_{S} M^\mu \norme{f} ) \dmu\big)^{1/p}\\
&\leq 2^{1/p} \norm{M^\mu \norme{f}}_{L^p(\mu)}\leq 2^{1/p}p'\norm{\norme{f}}_{L^p(\mu)}=2^{1/p}p'\norm{f}_{L^p(\mu;E)}.
\end{split}
\end{equation*}
\end{proof}
For each $S\in\cs$, we define the operator $P_S$ by setting
$$
P_Sf:=\sum_{\substack{Q\in\cd: \\\pi(Q)=S}} D_Qf
$$
for every locally integrable $f:\br^d\to E$.
\begin{lemma}[Properties of the operators $P_S$]\label{propertiesofps}For each $S\in\cs$, 
the operator $P_S$ has the following properties:
\begin{itemize}
\item[(i)]$$P_Sf=\sum_{S'\in\chs(S)} \langle f \rangle_{S'} 1_{S'}+f1_{\es(S)}-\langle f \rangle_S 1_S.$$
\item[(ii)]$P_Sf=f$ if and only if $f$ is supported on $S$, constant on each $S'\in\chs(S)$, and satisfies $\int_S f\dmu=0$.
\item[(iii)]$P_S^2=P_S$, $P_SP_{T}=0$ whenever $T\in\cs$ with $T\neq S$.
\item[(iv)]$\int g P_Sf \dmu =\int P_S g f \dmu$  for every $f\in L^p(E)$ and $g\in L^{p'}(E^*)$.
\item[(v)] $\normlp{P_Sf}\leq 2 \normlp{1_Sf}.$
\end{itemize}
\end{lemma}
\begin{proof}We prove the property (i), from which the other properties follow. On the one hand,
$$
f1_S=\sum_{Q: Q\subseteq S} D_Qf+\langle f \rangle_{S}1_S,
$$
on the other hand,
$$
f1_S=f1_{\es(S)}+\sum_{S'} f1_{S'}= f1_{\es(S)}+\sum_{S'\in\chs(S)} \big(\sum_{Q:Q\subseteq S'} D_Qf +\langle f \rangle_{S'}1_{S'}\big).
$$
Thus, by comparing,
$$
\sum_{Q: Q\subseteq S} D_Qf-\sum_{S'\in\chs(S)}\sum_{Q:Q\subseteq S'} D_Qf=f1_{\es(S)}+\sum_{S'\in\chs(S)} +\langle f \rangle_{S'}1_{S'}-\langle f \rangle_{S}1_S.
$$
Observing that $$\sum_{Q: Q\subseteq S} -\sum_{S'\in\chs(S)}\sum_{Q:Q\subseteq S'}=\sum_{Q: \pis(Q)=S}$$
completes the proof.
\end{proof}
The following variant of Pythagoras' theorem in the case $E=\br$ was proven by Katz and Pereyra \cite[Lemma 7]{pereyra1997} by using a multilinear estimate. We next give a different proof of the theorem, which extends it to an arbitrary Banach space $E$.
\begin{lemma}[Pythagoras' theorem for sparsely supported, piecewise constant functions]\label{pythagorassparse}Let $E$ be a Banach space. Let $1\leq p<\infty$. Let $\cs$ be a sparse collection of dyadic cubes. For each $S\in\cs$, assume that $f_S$ is a function that is supported on $S$ and constant on each $S'\in\chs(S)$. Then 
$$
\normlp{\sum_{S} f_S}\leq 3p\,\big(\sum_{S} \normlp{f_S}^p \big)^{1/p}.
$$
Moreover, the reverse estimate
$$
\big(\sum_{S} \normlp{f_S}^p \big)^{1/p}\leq 6p'\normlp{\sum_{S} f_S}
$$
holds if, in addition, one of the following conditions is satisfied:
$$
\mathrm{\text{(i)}} \int_S f_S\dmu=0, \quad \text{ or} \quad \text{(ii) } E=\br \text{ and } f_S\geq0,
$$
but may in general fail otherwise.
\end{lemma}
\begin{proof}First, we prove the direct estimate. 
By duality, it is equivalent to the estimate
$$
\int \sum_{S} f_S g \dmu \leq 3p\,\big(\sum_{S} \normlp{f_S}^p \big)^{1/p} \normlpp{g}.
$$ Since $f_S$ is supported on $S$, since $f_S$ is constant on $S'\in\chs(S)$, and since $S$ is partitioned by $\chs(S)$ and $\es(S)$,  we have
$$
\int \sum_{S} f_S g \,\dmu= \sum_{S} \int_S f_S g \dmu= \sum_{S} \sum_{S'\in\chs(S)}\langle f_S \rangle_{S'} \int_{S'} g\dmu+  \int \sum_{S} 1_{\es(S)} f_S g \dmu.
$$
We can estimate the second term by H\"older's inequality and the pairwise disjointness of the sets $\es(S)$,
\begin{equation*}
\begin{split}
\lvert \int  \sum_{S}1_{\es(S)} f_S g \dmu \rvert&\leq \normlp{\sum_{S}1_{\es(S)} f_S}\normlpp{g}\\
&=\big(\sum_{S}\normlp{1_{\es(S)} f_S}^p\big)^{1/p}\normlpp{g}.
\end{split}
\end{equation*}
We can estimate the first term as follows.
\begin{equation*}
\begin{split}
&\lvert \sum_{S} \sum_{S'\in\chs(S)}\langle f_S \rangle_{S'} \int_{S'} g \dmu \rvert\\
&\leq \sum_{S} \sum_{S'\in\chs(S)}\norme{ \langle  f_S \rangle_{S'}}\,\mu(  S' )^{1/p} \frac{\normes{ \int_{S'} g\dmu }}{\mu(S')} \mu(S' )^{1/p'}  \\
&\leq \Big( \sum_{S} \sum_{S'\in\chs(S)}\norme{ \langle  f_S \rangle_{S'}}^p\mu(  S' )\Big)^{1/p} \Big( \sum_{S} \sum_{S'\in\chs(S)} (\frac{\int_{S'} \normes{ g}\dmu }{\mu(S')} )^{p'} \mu( S' ) \Big)^{1/p'} \\
&\leq \Big( \sum_{S} \int_{S}\norme{ f_S }^p \dmu\Big)^{1/p}    \Big( \sum_{S} \sum_{S'\in\chs(S)} \langle \normes{g} \rangle_{S'}^{p'} \mu(S') \Big)^{1/p'}.
\end{split}
\end{equation*}
The proof of the direct estimate is completed by the special case of the dyadic Carleson embedding theorem, Lemma \ref{carlesonembedding}.

Next, we prove the reverse estimate under the assumption that $\int_S f_S=0$. By duality, this estimate is equivalent to the estimate
$$
\sum_{S} \int f_S g_S \dmu \leq 6p' \normlp{\sum_{S} f_S} \big( \sum_{S} \normlpp{g_S}^{p'}\big)^{1/p'}
$$
for arbitrary functions $g_S\in L^{p'}(E^*)$. 
By the properties of the operators $P_S$, Lemma \ref{propertiesofps}, we have that $\int f_S g_S \dmu=\int P_Sf_S g_S \dmu=\int P_S^2f_S g_S \dmu=\int P_Sf_S P_Sg_S \dmu=\int P_Sf_S \sum_{T}P_Tg_T \dmu=\int f_S \sum_{T}P_Tg_T \dmu$. Note that, although the functions $g_T$ are arbitrary, the functions $P_Tg_T$ satisfy the assumptions for the direct estimate: Each $P_Tg_T$ is supported on $T$, and constant on each $T'\in\chs(T)$. Thus, by H\"older's inequality and the direct estimate,
\begin{equation*}
\begin{split}
\sum_{S} \int f_S g_S \dmu&=\int \sum_{S} f_S \sum_{T\in\cs}P_T g_T \dmu\\
&\leq \normlp{\sum_{S} f_S} \normlpp{\sum_{T\in\cs}P_T g_T}\\
&\leq 3 p' \normlp{\sum_{S} f_S} \big( \sum_{T\in\cs}\normlpp{P_Tg_T}^{p'}\big)^{1/p'}\\
&\leq 6 p' \normlp{\sum_{S} f_S} \big( \sum_{T\in\cs}\normlpp{g_T}^{p'}\big)^{1/p'}.
\end{split}
\end{equation*}

Next, we prove the reverse estimate under the assumption that $E=\br$ and $f_S\geq0$. Since $f_S$ is supported on $S$, since $f_S$ is constant on $S'\in\chs(S)$, since $S$ is partitioned by $\chs(S)$ and $\es(S)$ and since $\mu(S')\leq 2 \mu(\es(S'))$ , we can write
\begin{equation*}
\begin{split}
\norm{f_S}^p_{L^p(\br)}&=\sum_{S'\in\chs(S)} \abs{\langle f_S \rangle_{S'}} ^p \mu( S') +\int 1_{\es(S)} \abs{f_S}^p\dmu\\
&\leq2 \sum_{S'\in\chs(S)}\langle \abs{f_S} \rangle_{S'} ^p \mu( \es(S') ) +\int 1_{\es(S)} \abs{f_S}^p \dmu\\
&= 2 \int \sum_{S'\in\chs(S)}1_{\es(S')} \abs{f_S}^p \dmu +\int 1_{\es(S)} \abs{f_S} ^p\dmu.
\end{split}
\end{equation*}
Summing over $S$ and taking into account that $\es(S)$ are pairwise disjoint yields
$$
\sum_{S} \norm{f_S}^p_{L^p(\br)} \leq 3 \int \big(\sum_{S} 1_{\es(S')} \abs{f_S}\big)^p \dmu.
$$Using the assumption that $f_S\geq 0$ completes the proof.

Lastly, we note that a simple example shows that the reverse estimate may in general fail. Indeed, let $S:=[0,1)$, $S_-:=[0,1/2)$,  $f_{S}:=1_{S_-}$, and $f_{S_-}:=-1_{S_-}$. Then $\norm{f_{S}}_{L^p(\br;\br)}^p+\norm{f_{S_-}}_{L^p(\br;\br)}^p=2\lvert S_-\rvert$ but $\norm{f_{S}+f_{S_-}}_{L^p(\br;\br)}=0$.
\end{proof}
\section{Decoupling of the sum of martingale differences}
Let $(X,\cf,\mu)$ be a $\sigma$-finite measure space. Let $(\ca_n)_{n=-\infty}^\infty$ be a refining sequence of countable partitions of $X$ into measurable sets of finite positive measure. Let $\ca:=\bigcup_{n=-\infty}^\infty \ca_n$. For each $K\in\ca_n$, let $\ch_\ca(K):=\{K'\in\ca_{n+1} : K'\subseteq K\}$. For each $K\in\ca$, let $f_K$ be a function that is supported on $K$ and constant on $K'\in\ch_\ca(A)$ and such that $\int_{K} f_K \mathrm{d}\mu =0$. Let $(Y_K,\cg_K,\nu_K)$ be the probability space such that $Y_K:=K$, $\cg_K$ is the $\sigma$-algebra generated by $\{K\}\cup\ch_\ca(K)$, and $\nu_K=\mu(K)^{-1} \mu\lvert_K$. Let $(Y,\cg,\nu)$ be the product probability space of the spaces $(Y_K,\cg_K,\nu_K)_{K\in\ca}$.

We notice that the sequence $(d_k)_{k=-\infty}^\infty$ with $d_k(x,y):=\sum_{K\in\ca_{k-1}} f_K(x)$ is a martingale difference sequence adapted to the filtration $(\cf_k)_{k=-\infty}^\infty$ generated by the refining sequence of partitions $(\ca_k)_{k=-\infty}^\infty$. Conversely, each martingale difference sequence $(d_k)_{k=-\infty}^\infty$ adapted to the filtration $(\cf_k)_{k=-\infty}^\infty$ can be written as $d_k:=\sum_{K\in\ca_{n-1}}f_K$, where for each $K\in\ca_{k-1}$ the function $f_K$ is defined by $f_K:=1_Kd_k=\sum_{\substack{K'\in\ca_{k}:K'\subseteq K}} \langle d_k \rangle_{K'} 1_{K'}.$

A variant of the following decoupling equality was proven by Hyt\"onen \cite[Theorem 6.1]{hytonen2008} as a corollary of McConnell's \cite[Theorem 2.2]{mcconnell1989} decoupling inequality for UMD-valued martingale difference sequences.

\begin{theorem}[Decoupling equality for piecewise constant, cancellative functions]\label{martingaledecoupling}Let $1<p<\infty$. Let $E$ be a UMD space. Then 
\begin{equation*}
\begin{split}
&\frac{1}{\beta_p(E)} \big(\be \norm{\sum_{K\in\ca} \epsilon_K 1_K(x)f_K(y_K)}_{L^p(\mathrm{d}\mu(x)\times\mathrm{d}\nu(y);E)}^p\big)^{1/p}\\
&\leq \norm{\sum_{K\in\ca} f_K(x)}_{L^p(\mathrm{d}\mu(x);E)}\\
&\leq \beta_p(E) \big(\be \norm{\sum_{K\in\ca} \epsilon_K 1_K(x)f_K(y_K)}_{L^p(\mathrm{d}\mu(x)\times\mathrm{d}\nu(y);E)}^p\big)^{1/p}.
\end{split}
\end{equation*}
\end{theorem}

Here we give another proof of the equality: Roughly speaking, we construct auxiliary martingale differences $u_K(x,y_K)$ and $v_K(x,y_K)$ such that $f_K(x)=u_K(x,y_K)+v_A(x,y_A)$ and $1_K(x)f_K(y_K)=u_K(x,y_K)-v_K(x,y_K)$, from which the decoupling equality follows by the definition of the UMD property. Let $d_k$ be a martingale difference sequence adapted to the filtration $\cf_k$. We write
$$
d_k(x,y)=\sum_{K\in\ca_{k-1}} 1_K(x) d_k(x) 1_K(y_K),
$$
and$$ \tilde{d}_k(x,y):=\sum_{K\in\ca_{k-1}} 1_K(x)d_k(y_K)1_K(y_K).
$$

\begin{proposition}[Constructing auxiliary martingale differences]
%
There exists a martingale difference sequence $(u_k)_{k\in\frac{1}{2}\bz}$ on the product measure space $(X\times Y,\cf\times\cg,\mu\times \nu)$ such that
$$
d_k=u_k+u_{k+1/2}, \quad\text{and} \quad\tilde{d}_k=u_k-u_{k+1/2}.
$$
\end{proposition}
\begin{proof}
Let $d_k$ be a martingale difference sequence $d_k$ adapted to the filtration $\cf_k$ generated by a refining sequence of partitions $\ca_k$. The $\cf_k$-measurability of $d_k$ means that $d_k$ equals to a constant $\ave{d_k}_K$ on $K\in\ca_k$. Thus, we can write
\begin{equation*}
  d_k=\sum_{K\in\ca_{k-1}}1_K d_k
        =\sum_{K\in\ca_{k-1}} 1_K \sum_{\substack{K'\in\ca_k\\ K'\subseteq K}}1_{K'} d_k,
        =\sum_{K\in\ca_{k-1}} 1_K \sum_{\substack{K'\in\ca_k\\ K'\subseteq K}}1_{K'} \ave{d_k}_{K'}.
\end{equation*}
The martingale difference property $\Exp[d_k|\cf_{k-1}]=0$ means that for every $K\in\ca_{k-1}$ we have
\begin{equation*}
  \int_K d_k\dmu=\sum_{\substack{K'\in\ca_k\\ K'\subseteq K}}\ave{d_k}_{K'}\mu(K')=0.
\end{equation*}

First, we consider a fixed $K\in\ca_{k-1}$. Let $\nu_K$ be the the measure $\nu_K:=\mu(K)^{-1}\mu\lvert_K$ resctricted to the sub-$\sigma$-algebra $\cg_K$ that is generated by the collection $\{K\}\cup\{K'\in \ca_{k} : K'\subseteq K\}$. Note that the functions
\begin{equation*}
  d_K(x,y_K):=1_K(x)d_k(x)1_K(y_K)
  =\sum_{\substack{A,B\in\ca_k \\ A,B\subseteq K}}\ave{d_k}_A 1_A(x)1_B(y_K)
\end{equation*}
and
\begin{equation*}
 \tilde{d}_K(x,y_K):= 1_K(x)d_k(y_K)1_K(y_K)
  =\sum_{\substack{A,B\in\ca_k \\ A,B\subseteq K}}\ave{d_k}_B 1_A(x)1_B(y_K)
\end{equation*}
are equally distributed in the measure space $(\br^d\times K,\mu\times \nu_K, \cf\times\cg_K)$, which is to say that the functions take the same values in sets of equal measure. We define the functions $u_K(x,y_K)$ and $v_K(x,y_K)$ by the pair of equations 
\begin{equation*}
\begin{split}
  d_K(x,y_K) &=:u_K(x,y_K)+v_K(x,y_K), \\
  \tilde{d}_K(x,y_K)&=:u_K(x,y_K)-v_K(x,y_K).
\end{split}
\end{equation*}
Therefore, the function $u_K(x,y_K)$ can be written out as
\begin{equation*}
\begin{split}
  u_K(x,y_K)
  &=\frac{1}{2}(d_K(x,y_K)+\tilde{d}_K(x,y_K)) \\
  &=\sum_{\substack{A,B\in\ca_k \\ A,B\subseteq K}}\frac{1}{2}(\ave{d_k}_A+\ave{d_k}_B) 1_A(x)1_B(y_K) \\
  &=\sum_{\substack{A\in\ca_k\\ A\subseteq K}}\ave{d_k}_A 1_{A\times A}(x,y_K) \\
    &\qquad+\sum_{\substack{A,B\in\ca_k\\ A,B\subseteq K; A<B}}\frac{1}{2}(\ave{d_k}_A+\ave{d_k}_B)1_{A\times B\cup B\times A}(x,y_K),
\end{split}
\end{equation*}
where in the last step we introduced some order among the finite family
\begin{equation*}
  \{A\in\ca_k:A\subseteq K\}
  =\{A_i\}_{i=1}^{I(K)},
\end{equation*}
and defined $A<B$ if and only if $A=A_i$, $B=A_j$, and $i<j$. The function $v_K(x,y_K)$ can be written out as 
\begin{equation*}
\begin{split}
  v_K(x,y_K)
  &=\frac{1}{2}(d_K(x,y_K)-\tilde{d}_K(x,y_K)) \\
  &=\sum_{\substack{A,B\in\ca_k \\ A,B\subseteq K}}\frac{1}{2}(\ave{d_k}_A-\ave{d_k}_B) 1_A(x)1_B(y_K) \\
  &=\sum_{\substack{A,B\in\ca_k\\ A,B\subseteq K; A<B}}\frac{1}{2}(\ave{d_k}_A-\ave{d_k}_B)(1_{A\times B}(x,y_K)-1_{B\times A}(x,y_K)).
\end{split}
\end{equation*}
\begin{figure}[b]
    \centering
  \includegraphics[width=0.95\textwidth]{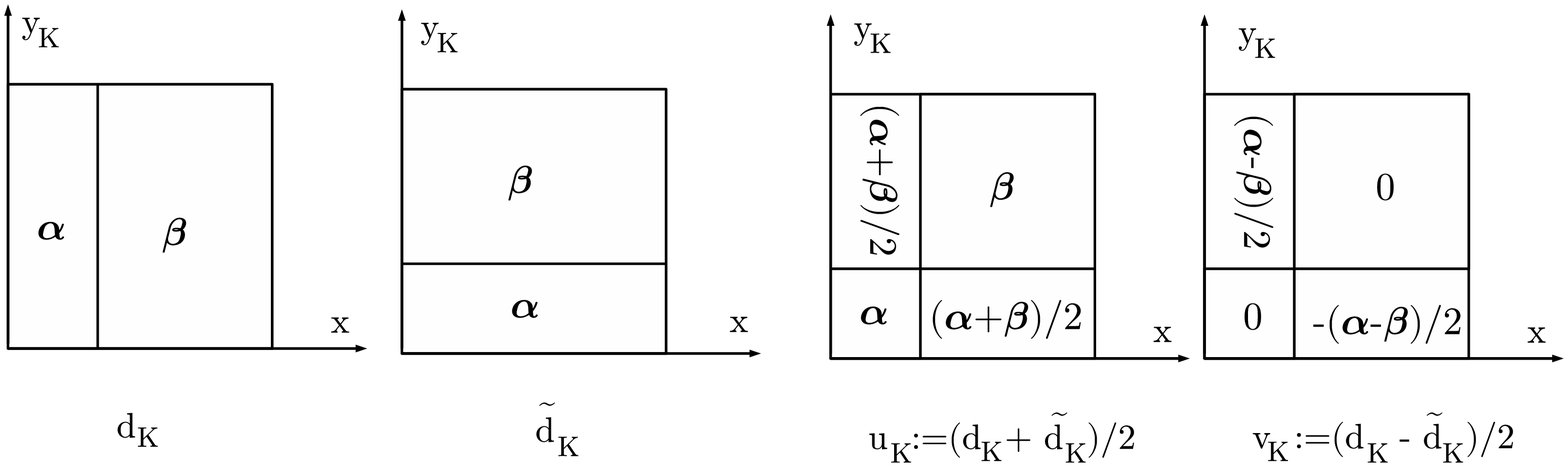}
    \caption{Functions $d_K$, $\tilde{d}_K$, $u_K$, and $v_K$.}
    \label{fig:decoupling}
\end{figure}

Next, we define the product measure space. For each $k\in\bz$ and $K\in\ca_{k-1}$, let $(Y_K,\cg_K,\nu_K)$ be the probability space such that  $Y_K:=K$, $\cg_K$ is the $\sigma$-algebra generated by $\{K\}\cup\{K'\in\ca_{k} : K'\subseteq K\}$, and $\nu_K=\mu(K)^{-1} \mu\lvert_K$. Let $(Y,\cg,\nu)$ be the product probability space of the spaces $(Y_K,\cg_K,\nu_K)_{K\in\ca}$. We recall that the product space $Y$ is the Cartesian product $Y=\prod_{K\in\ca} Y_K$,  the product $\sigma$-algebra $\cg$ (in the case of a countable index set) is the collection
$$
\cg:=\Big\{ \prod_{K\in\ca}^* G_K : G_K\in\cg_K \Big\},
$$
and the product measure $\nu$ is the unique measure on $(Y,\cg)$ that satisfies
$$
\nu( \prod_{K\in\ca}^* G_K)=\prod_{K\in\ca}^* \nu_K(G_K),
$$ 
where $*$ in the product indicates that for only finitely many $G_K$ we have $G_K \neq Y_K$. 

Next, we prove that the sequence $(\ldots,u_k,u_{k+1/2},u_{k+1},u_{k+1+1/2},\ldots)$ defined by
$$
u_k(x,y):=\sum_{K\in\ca_{k-1}} u_K(x,y_K),
$$
and
$$
v_k(x,y):=u_{k+1/2}(x,y):=\sum_{K\in\ca_{k-1}} v_K(x,y_K)
$$
 is a martingale difference sequence in the measure space $(X\times Y, \cf\times\cg, \mu\times\nu)$.
Proving this is based on the following observations:
\begin{itemize}
  \item[(a)] For each $K\in\ca_{k-1}$, the function $u_K(x,y_K)$ depends on $x$ and $y_K$ ``in a symmetric way'' (see Figure \ref{fig:decoupling});
   \item[(b)] For each $K\in\ca_{k-1}$, the function $v_K(x,y_K)$ depends on $x$ and $y_K$ ``in an anti-symmetric way'' (see Figure \ref{fig:decoupling});
  \item[(c)] The function $u_K$ averages to zero on the set $K\times K$ because $d_k$ itself is a martingale difference; Indeed,
\begin{equation*}\label{eq:mdsUk}
\begin{split}
 &\int_K\int_K u_K(x,y_K)\dmu(x)\dnu_K(y_K)  \\
  &=\frac{1}{\mu(K)}\Big[\sum_{\substack{A\in\ca_k\\ A\subseteq K}}\ave{d_k}_A \mu(A)\mu(A)
    +\sum_{\substack{A,B\in\ca_k\\ A,B\subseteq K\\ A<B}}\frac{1}{2}(\ave{d_k}_A+\ave{d_k}_B) \cdot2\mu(A)\mu(B)\Big] \\
  &=\frac{1}{\mu(K)}\sum_{\substack{A\in\ca_k\\ A\subseteq K}}\ave{d_k}_A \mu(A)\sum_{\substack{B\in\ca_k\\ B\subseteq K}}\mu(B) 
   =\sum_{\substack{A\in\ca_k\\ A\subseteq K}}\ave{d_k}_A \mu(A)
   =\int_K d_k(x)\dmu(x) =0.
\end{split}
\end{equation*}
  \item[(d)] The function $v_K$ takes equal positive and negative values on two halves of the symmetric sets $A\times B\cup B\times A$ with $A<B$, whereas the function $u_K$ takes equal values on both the halves. Moreover, the function $v_K$ takes zero value on the symmetric sets $A\times A$. Thus, for any function $\phi(u_K)$, we have
\begin{equation*}
\begin{split}
&\int_{K\times K} v_K \phi(u_K) \dmu\dnu_K\\
&=\sum_{\substack{A,B\in\ca_k\\ A,B\subseteq K; A<B}} \bigg(\int_{A\times B}v_K\dmu \dnu_K+\int_{B\times A}v_K\dmu\dnu_K\bigg)  \langle \phi(u_K)\rangle_{A\times B\cup B\times A }^{\mu\times\nu_K}=0,
\end{split}
\end{equation*}
where the average $\langle \phi(u_K)\rangle_{A\times B\cup B\times A }^{\mu\times\nu_K}$ denotes the constant value of $\phi(u_K)$ on the set $A\times B\cup B\times A$.
\end{itemize}
We define the filtration $(\cu_k)_{k\in\frac{1}{2}\bz}$  as follows. For each $k\in\bz$, we define the $\sigma$-algebra $\cu_k$ as the $\sigma$-algebra generated by the functions  $\{u_l\}_{l:l\leq k}$ and $\{1_K\}_{K\in\ca_{k-1}}$, and, similarly, the $\sigma$-algebra $\cu_{k+1/2}$ as the $\sigma$-algebra generated by the functions  $\{u_l\}_{l:l\leq k+1/2}$ and $\{1_K\}_{K\in\ca_{k-1}}$. We note that the functions $(1_K)_{K\in\ca_{k-1}}$ are included for technical reasons: They ensure that each $\cu_{k}$, with $k\in\frac{1}{2}\bz$, is $\sigma$-finite so that taking the conditional expectation with respect to it makes sense. Now, by definition, each $u_k$ is measurable with respect to $\cu_k$, and, furthermore, $(\cu_k)_{k\in\frac{1}{2}\bz}$ is a filtration. Next, we check that $\be[u_k|\,\cu_{k-1/2}]=0$, which is equivalent to checking that
$$
\int_{X\times Y} u_k\phi\big(\{u_l\}_{l:l\leq k-1/2},\{1_K\}_{K\in\ca_{[k-1/2]-1}}\big)\dmu\dnu=0
$$
for all functions $\phi(\ldots,u_l,\ldots,u_{k-1/2}, 1_{K},\ldots)=:\phi\big(\{u_l\}_{l:l\leq k-1/2},\{1_K\}_{K\in\ca_{[k-1/2]-1}}\big)$.

First, we check that $\be[u_k|\cu_{k-1/2}]=0$ for $k\in\bz$. We have
\begin{equation*}
\begin{split}
&\int_{X\times Y} u_k\phi\big(\{u_l\}_{l:l\leq k-1/2},\{1_K\}_{K\in\ca_{k-2}}\big)\dmu\dnu\\
&=\sum_{K\in\ca_{k-1}} \int_{K\times Y} u_K\phi\big(\{u_l\}_{l:l \leq k-2},\{1_K\}_{K\in\ca_{k-2}}\big) \dmu\dnu.
\end{split}
\end{equation*}
We note that each of the functions $\{u_l(x,y)\}_{l:l\leq k-1/2}$, and $\{1_{K'}(x)\}_{K'\in\ca_{k-2}}$ is constant with respect to $x\in K\in\ca_{k-1}$; We denote these constant values by their averages. Moreover, $u_K(x,y)$ depends on $y$ only via $y_K$. Therefore, by pulling out the constant, and integrating out the independent variables, we obtain
\begin{equation*}
\begin{split}
&\int_{K\times Y} u_K\phi\big(\{u_l\}_{l:l\leq k-1/2},\{1_K\}_{K\in\ca_{k-2}}\big) \dmu\times\dnu\\
&=\int_{K\times Y_K} u_K \dmu \dnu_K  \\
&\cdot\int_{\prod_{\substack{K'\in\ca:\\K'\neq K}}Y_{K'}}  \phi\big(\{\langle u_l \rangle_K^\mu\}_{l:l\leq k-1/2},\{\langle 1_{K'} \rangle_K^\mu\}_{K'\in\ca_{k-2}}\big) \prod_{\substack{K'\in\ca:\\K'\neq K}}\dnu_{K'}.
\end{split}
\end{equation*}
The observation (c) states that $\int_{K\times Y_K} u_K \dmu\dnu_K=0$.

Finally, we check that $\be[u_{k+1/2}|\cu_{k}]=0$ for $k\in\bz$. Again, we have
\begin{equation*}
\begin{split}
&\int_{X\times Y} u_{k+1/2}\phi\big(\{u_l\}_{l:l\leq k},\{1_K\}_{K\in\ca_{k-1}}\big)\dmu\dnu\\
&=\sum_{K\in\ca_{k-1}} \int_{K\times Y} v_K\phi\big(\{u_l\}_{l:l\leq k},\{1_K\}_{K\in\ca_{k-1}}\big) \dmu\dnu.
\end{split}
\end{equation*}
We note that each of the functions $\{u_l(x,y)\}_{l:l\leq {k-1}}$, and $\{1_{K'}(x)\}_{K'\in\ca_{k-1}}$ is constant with respect to $x\in K\in\ca_{k-1}$; We denote these constant values by their averages. Furthermore, $u_k(x,y)=u_K(x,y)+ \sum_{\substack{K'\in\ca_{k-1}:\\K'\neq K}}u_{K'}(x,y)=u_K(x,y)$ for $x\in K\in\ca_{k-1}$. Moreover, $v_K(x,y)$ depends on $y$ only via $y_K$. Therefore, again by pulling out the constant, and integrating out the independent variables, we obtain
\begin{equation*}
\begin{split}
&\int_{K\times Y} v_K\phi\big(\{u_l\}_{l:l\leq k},\{1_K\}_{K\in\ca_{k-1}}\big) \dmu \dnu\\
&=\int_{K\times Y_K} v_K \Phi_K (u_K)\dmu \dnu_K,
\end{split}
\end{equation*}
where 
\begin{equation*}
\begin{split}
&\Phi_K (u_K):=\int_{\prod_{\substack{K'\in\ca:\\K'\neq K}}Y_{K'}} \phi\big(\{\langle u_l \rangle^\mu_K\}_{l:l\leq k-1}, \{u_K\} ,\{\langle 1_{K'} \rangle^\mu_K\}_{K'\in\ca_{k-1}}\big)  \prod_{\substack{K'\in\ca:\\K'\neq K}}\dnu_{K'}.
\end{split}
\end{equation*}
The observation (d) states that $\int_{K\times Y_K} v_K \Phi_K (u_K)\dmu \dnu_K=0$.
\end{proof}

\section{Vector-valued dyadic shifts are bounded}\label{section_dyadicshifts}
%
%

Let $L:=\max\{i,j\}+1$.  By picking every $L$:th length scale, we decompose the collection $\cd$ of dyadic cubes to subcollections $\mathcal{D}_{l \;\mathrm{mod} \,L}$, with $l=0,\ldots,L-1$, such that for every $K\in \mathcal{D}_{l \;\mathrm{mod} \,L}$ we have that both $D^i_Kf$ and $D^j_Kg$ are constant on $K'\in \ch_{\mathcal{D}_{l \;\mathrm{mod} \,L}}(K)$ and have zero average on $K$. More specifically, for each $l=1,\ldots,L-1$, let $\mathcal{D}_{l \;\mathrm{mod} \,L}=\bigcup_{k=-\infty}^\infty \{K\in \cd : \ell(K)=2^{-k L+l} \}$.
Then
\begin{equation*}
\begin{split}
&S^{ji}f=\sum_{K\in\cd} D^j_K A^{ji} D^i_Kf =\sum_{l=0}^{L-1} \sum_{K\in\mathcal{D}_{l \;\mathrm{mod} \,L}}  D^j_K A^{ji} D^i_Kf. \\
\end{split}
\end{equation*}


This decomposition is done in order to decouple by using Theorem \ref{martingaledecoupling}. From now on we consider a fixed $l$. We write $e_K:=D^j_K A^{ji} D^i_Kf$. Let $\mathrm{d}\nu_K(x)=\lvert K \rvert^{-1} 1_K(x)\mathrm{d}x$ be the Lebesgue measure restricted and normalized to the dyadic cube $K$. Let $\nu$ denote the product measure $\bigtimes_{K\in\ca}\nu_K$ on the product space $Y:=\Pi_{K\in\ca} K$. By Theorem \ref{martingaledecoupling},
\begin{equation*}
\begin{split}
&\norm{\sum_{K\in\mathcal{D}_{l \;\mathrm{mod} \,L}} 1_K(x)e_K(x)}^p_{L^p(\mathrm{d}x;E)}\\
&\leq \beta_p(E)^p \be \norm{\sum_{K\in\mathcal{D}_{l \;\mathrm{mod} \,L}} \epsilon_K 1_K(x)e_K(y_K)}^p_{L^p(\mathrm{d}x\times\mathrm{d}\nu(y);E)}.
\end{split}
\end{equation*}
We write $e_K(y_K)=D^j_K  A_K D^i_K f(y_K)=:D^j_K g_K(y_K)$. By using Lemma \ref{independentfunctions} together with the fact that $D^j_K$ is a difference of two conditional expectations, we obtain
\begin{equation*}
\begin{split}
&\norm{\epsilon_K 1_K(x) \sum_{K\in\mathcal{D}_{l \;\mathrm{mod} \,L}} D^j_K g_K(y_K) }^p_{L^p(\mathrm{d}\nu(y);E)}\\
&\leq 2^p  \norm{\sum_{K\in\mathcal{D}_{l \;\mathrm{mod} \,L}}  \epsilon_K 1_K(x)g_K(y_K)}^p_{L^p(\mathrm{d}\nu(y);E)}.
\end{split}
\end{equation*}
We write $g_K(y_K)=A_K D^i_K f(y_K)=: A_K f_K(y_K)$. By introducing an independent copy $(\tilde{Y},\tilde{\nu})$ of the probability space $(Y,\nu)$, we write
\begin{equation*}
\begin{split}
A_K f_K(y_K)&=\frac{1_K(y_K)}{\lvert K \rvert}\int_K a_K(y_K,x') f_K(x')\mathrm{d}x'\\
&=\int_{\tilde{Y}} 1_K(y_K)a_K(y_K,\tilde{y}_K) f_K(x')\mathrm{d}\tilde{\nu}(\tilde{y}).
\end{split}
\end{equation*}
By Jensen's inequality,
\begin{equation*}
\begin{split}
&\Big\lvert\int_{\tilde{Y}} \sum_{K\in\mathcal{D}_{l \;\mathrm{mod} \,L}}\epsilon_ K 1_K(x) a_K(y_K,\tilde{y}_K) f_K(\tilde{y}_K)\,\mathrm{d}\tilde{\nu}(\tilde{y})\Big\rvert_E^p\\
&\leq \int_{\tilde{Y}}  \norme{\sum_{K\in\mathcal{D}_{l \;\mathrm{mod} \,L}} \epsilon_ K 1_K(x) a_K(y_K,\tilde{y}_K) f_K(\tilde{y}_K)}^p \,\mathrm{d}\tilde{\nu}(\tilde{y}).
\end{split}
\end{equation*}
Since the family of operators $ \{a_K(x,x') : K\in\cd,x\in K, x'\in K \}$ is $R$-bounded, we have
$$
\be \norme{\sum_{K\in\mathcal{D}_{l \;\mathrm{mod} \,L}}\epsilon_ K 1_K(x) a_K(y_K,\tilde{y}_K) f_K(\tilde{y}_K)}^p \leq \mathcal{R}_p(\{a\})^p \,\be \norme{\sum_{K\in\mathcal{D}_{l \;\mathrm{mod} \,L}}\epsilon_ K 1_K(x) f_K(\tilde{y}_K)}^p.
$$
Altogether, we have obtained that
\begin{equation*}
\begin{split}
&\norm{\sum_{K\in\mathcal{D}_{l \;\mathrm{mod} \,L}} 1_K(x)e_K(x)}^p_{L^p(\mathrm{d}x;E)}\\
&\leq (2\beta_p(E)  \mathcal{R}_p(\{a\}))^p \int_{Y} \be \norm{ \sum_{K\in\mathcal{D}_{l \;\mathrm{mod} \,L}}\epsilon_K 1_K(x)D^i_K f(\tilde{y}_K)}^p_{L^p(\mathrm{d}x\times \mathrm{d}\tilde{\nu}(\tilde{y});E)} \mathrm{d}\nu(y)\\
&=(2\beta_p(E)  \mathcal{R}_p(\{a\}))^p \be \norm{ \sum_{K\in\mathcal{D}_{l \;\mathrm{mod} \,L}}\epsilon_K 1_K(x)D^i_K f(\tilde{y}_K)}^p_{L^p(\mathrm{d}x\times \mathrm{d}\tilde{\nu}(\tilde{y});E)}.
\end{split}
\end{equation*}
Since $D^i_KD^m_K=0$ whenever $i\neq m$, we can write $D^i_K=D^i_K \sum_{m=0}^{L-1} D^m_K$. By using Lemma \ref{independentfunctions} together with the fact that $D^j_K$ is a difference of two conditional expectations, we obtain 
\begin{equation*}
\begin{split}
&\be \norm{ \sum_{K\in\mathcal{D}_{l \;\mathrm{mod} \,L}}\epsilon_K 1_K(x)D^i_K \sum_{m=0}^{L-1} D^m_K f(\tilde{y}_K)}^p_{L^p(\mathrm{d}x\times \mathrm{d}\tilde{\nu}(\tilde{y});E)}\\
&\leq 2^p\be \norm{ \sum_{K\in\mathcal{D}_{l \;\mathrm{mod} \,L}}\epsilon_K 1_K(x)\sum_{m=0}^{L-1} D^m_K f(\tilde{y}_K)}^p_{L^p(\mathrm{d}x\times \mathrm{d}\tilde{\nu}(\tilde{y});E)}
\end{split}
\end{equation*}

We have that $\sum_{m=0}^{L-1} D^m_K f$ is constant on $K'\in \ch_{\mathcal{D}_{l \;\mathrm{mod} \,L}}(K)$ and has zero average on $K$. Therefore, by removing the decoupling using Theorem \ref{martingaledecoupling}, we obtain
\begin{equation*}
\begin{split}
&\be \norm{ \sum_{K\in\in\mathcal{D}_{l \;\mathrm{mod} \,L}}\epsilon_K 1_K(x)\sum_{m=0}^{L-1} D^m_K f(\tilde{y}_K)}^p_{L^p(\mathrm{d}x\times \mathrm{d}\tilde{\nu}(\tilde{y});E)}\\
&\leq \norm{\sum_{K\in\in\mathcal{D}_{l \;\mathrm{mod} \,L}} \sum_{m=0}^{L-1} D^i_K f(x)}^p_{L^p(\mathrm{d}x;E)} = \norm{\sum_{K\in\cd}D_K f}_{L^p(\br^d;E)}^p =\norm{f}_{L^p(\br^d;E)}^p.
\end{split}
\end{equation*}
The proof is completed.

\section{Sufficient condition for the boundedness of dyadic paraproducts}\label{section_paraproduct}
From the fact that $\normlp{\langle f \rangle_{Q_0}1_{Q_0}}\to 0$ as $\ell(Q_0)\to \infty$, it follows that the functions of the form $f:=\sum_{Q_0} f_{Q_0}:=\sum_{Q_0} (f-\langle f \rangle_{Q_0})1_{Q_0}$, where $Q_0$ are disjoint dyadic cubes, are dense in $L^p(E)$. Hence it suffices to prove the estimate  
$$
\norm{\sum_{Q\in \cd : Q\subseteq Q_0} D_Qb\, \langle f \rangle_Q }_{L^p(\br^d;E)}\leq 6\cdot 2^{d} pp'\beta_p(E)^2 \beta_p(\ct) \norm{b}_{\text{BMO}_p(\br^d;\ct)} \norm{f}_{L^p(\br^d;E)}
$$
uniformly for all $Q_0\in\cd$. Now, we fix a dyadic cube $Q_0$. Let $\cd(Q_0):=\{Q\in \cd : Q\subseteq Q_0\}$. Let $\cs:=\cs(Q_0)\subseteq \cd(Q_0)$ be a sparse collection that contains the cube $Q_0$. For each $Q\in\cd$, let $\pis(Q)$ denote the minimal dyadic cube $S\in\cs$ such that $S\supseteq Q$. We rearrange the summation as $\sum_{Q\in\cd(Q_0)}=\sum_{S\in\cs} \sum_{\substack{Q\in\cd(Q_0) :\\ \pi(Q)=S}}$.  By the variant of Pythagoras' theorem, Theorem \ref{pythagorassparse}, we obtain
\begin{equation*}
\begin{split}
&\normlp{\sum_{Q\in\cd(Q_0)} D_Qb\, \langle f \rangle_Q }=\normlp{\sum_{S\in\cs}  \sum_{\substack{Q\in\cd(Q_0) :\\ \pi(Q)=S}} D_Qb\,\langle f \rangle_Q }\\
&\leq 3 p \big(\sum_{S\in\cs}  \normlp{ \sum_{\substack{Q\in\cd(Q_0) :\\ \pi(Q)=S}} D_Qb \langle f \rangle_Q}^p\big)^{1/p}.
\end{split}
\end{equation*}
It remains to choose the sparse collection $\cs$ so that 
\begin{equation}\label{aim}\normlp{ \sum_{\substack{Q\in\cd(Q_0) :\\ \pi(Q)=S}} D_Qb \langle f \rangle_Q}\leq C_{b,E,p,d} \langle \norme{f} \rangle_S \lvert S \rvert^{1/p},
\end{equation} 
which, by the special case of the dyadic Carleson embedding theorem, Lemma \ref{carlesonembedding}, completes the proof by the estimate
\begin{equation*}
\begin{split}
\big(\sum_{S\in\cs}  \normlp{ \sum_{\substack{Q\in\cd(Q_0) :\\ \pi(Q)=S}} D_Qb \langle f \rangle_Q}^p\big)^{1/p} \leq C_{b,E,p,d} \big(\sum_{S\in\cs} \langle \norme{f} \rangle_S^p \lvert S \rvert \big)^{1/p}\leq  C_{b,E,p,d} 2 p' \normlp{f}.
\end{split}
\end{equation*}

Next, we choose the collection $\cs$ so that the estimate \eqref{aim} is satisfied. For each $S\in\cd$, let $\chs(S)$ be the collection of all the maximal dyadic subcubes $S'\subsetneq S$ such that
\begin{equation}\label{stoppingeq}
\langle \norme{f} \rangle_{S'}> 2 \langle \norme{f} \rangle_S.
\end{equation}By the dyadic nestedness and maximality, the collection $\chs(S)$ is pairwise disjoint.
We define recursively $\cs_0:=\{Q_0\}$  and $\cs_{n+1}:=\bigcup_{S\in\cs_n} \chs(S)$. Let $\cs:=\bigcup_{n=0}^\infty \cs_n$. We define the pairwise disjoint sets $\es(S):=S\setminus \bigcup_{S'\in\chs(S)}S'$. By construction,
$$
\sum_{S'\in\chs(S)} \lvert S'\rvert \leq  \frac{1}{2} \lvert S \rvert,
$$ 
which is to say that $\lvert \es(S) \rvert \geq \frac{1}{2} \lvert S \rvert$. Hence the collection $\cs$ is sparse. 

Next, we check that $\int_Qf\dx= \int_Q f_S \dx$ for $f_S:=f1_{\es(S)}+\sum_{S'\in\chs(S)} \langle f \rangle_{S'}1_{S'}$ whenever $\pi(Q)=S$. Firstly, the set $Q$ is partioned by $\es(S)\cap Q$ and $\{S'\in\chs(S): S'\cap Q\neq \emptyset\}$. Secondly, by the dyadic nestedness, $S'\cap Q\neq \emptyset$ implies that either $Q\subseteq S'$ or $S'\subsetneq Q$. The alternative $Q\subseteq S'$ is excluded because $\pis(S)=Q$ means that $S$ is the minimal $S''\in\cs$ such that $Q\subseteq S''$. Hence $S'\subsetneq Q$ for all $S'\in\chs(S)$ with $S'\cap Q\neq \emptyset$. Therefore 
\begin{equation*}
\begin{split}
&\int_Qf\dx=\int_Q f1_{\es(S)}+\sum_{S'\in\chs(S):S'\subseteq Q} \int_{S'}  f \dx\\
&=\int_Q f1_{\es(S)}+\sum_{S'\in\chs(S):S'\subseteq Q}\int_{Q} \langle f \rangle_{S'} 1_{S'}\dx=\int_Q f_S \dx.
\end{split}
\end{equation*}

Next, we check that $\norme{f_S}\leq 2\cdot 2^{d} \langle \norme{f} \rangle_S$ almost everywhere. First, let $x\in\es(S)$. Then, by construction, for all $Q\in\cd$ such that $Q\ni x$ we have $ \langle \norme{f} \rangle_Q\leq 2 \langle \norme{f} \rangle_S$. Therefore, by the Lebesgue differentiation theorem, $\norme{f(x)}\leq 2 \langle \norme{f} \rangle_S$ for almost every such $x$. Let $S'\in\ch(S)$. By the maximality of $S'$, the dyadic parent $\hat{S'}$ of $S'$ satisfies the opposite $\langle \norme{f} \rangle_{\hat{S'}}\leq 2 \langle \norme{f} \rangle_S$ of the inequality \eqref{stoppingeq}. By doubling, $\langle \norme{f} \rangle_{S'}\leq 2^d \langle \norme{f} \rangle_{\hat{S'}}$. Altogether, $\langle \norme{f} \rangle_{S'}\leq 2\cdot 2^d \langle \norme{f} \rangle_{S'}$.

Altogether, we have that
$$
\normlp{ \sum_{\substack{Q\in\cd(Q_0) :\\ \pi(Q)=S}} D_Qb \langle f \rangle_Q}=\normlp{ \sum_{\substack{Q\in \cd(Q_0) :\\ \pi(Q)=S}} D_Qb \langle f_S \rangle_Q}
$$
with $\norm{f_S}_{L^\infty(E)}\leq 2\cdot 2^d \langle \norme{f} \rangle_S$. The proof is completed by Lemma \ref{linfinityestimate}. 
\begin{lemma}\label{linfinityestimate} Let $1<p<\infty$. Let $E$ be a UMD space.  Assume that $\ct$ is a UMD subspace of $\cl(E)$. Let $S$ be a dyadic cube and let $\cq(S)$ be a collection of dyadic subcubes of $S$. Then
$$
\norm{ \sum_{\substack{Q\in\cq(S)}} D_Qb \langle f \rangle_Q}_{L^p(E)}\leq \beta_p(E)^2 \beta_p(\ct) \norm{b}_{\text{BMO}_p(\ct)} \norm{f}_{L^\infty(S;E)} \lvert S \rvert^{1/p}
$$
for any $f\in L^\infty(S;E)$ and $b\in\text{BMO}_p(\br^d;\ct)$.
\end{lemma}

\begin{proof}[Proof without the decoupling equality]By the UMD property and the Kahane contraction principle, we obtain
$$
\norm{\sum_{Q\in \cq(S)} D_Qb \langle f \rangle_Q}_{L^p(E)}^p\leq \beta_p(E)^p \be \norm{\sum_{Q: Q \subseteq S} \epsilon_Q D_Qb \langle f \rangle_Q 1_Q}_{L^p(E)}^p. 
$$ 
We expand
$$
D_Qb=\sum_{\eta\in\{0,1\}^d\setminus\{0\}} \langle b, h_Q^\eta\rangle h_Q^\eta, 
$$
where, for each $Q=I_1\times\cdots\times I_d$ and $\eta=(\eta_1,\ldots,\eta_d)\in \{0,1\}^d$, we have $h^\eta_Q=h_{I_1}^{\eta_1}\cdots h_{I_d}^{\eta_d}$ with $h_I:=h^1_I:=\frac{1}{\sqrt{\lvert I \rvert}} (1_{I_{\text{left}}}-1_{I_{\text{right}}})$ and $h_I^0=\frac{1}{\sqrt{\lvert I \rvert}}1_I$.  Therefore
$$
\big(\be \norm{\sum_{Q\subseteq S} \epsilon_Q D_Qb \langle f \rangle_Q 1_Q}_{L^p(E)}^p\big)^{1/p}\leq  \sum_{\eta\in\{0,1\}^d\setminus\{0\}} \big(\be \norm{\sum_{Q\subseteq S} \epsilon_Q h^\eta_Q \langle \langle b, h^\eta_Q \rangle f \rangle_Q 1_Q}_{L^p(E)}^p\big)^{1/p}.
$$

Next, we consider a fixed $\eta$. We observe that, at each point $x\in \br^d$, we have $h_Q^\eta(x)=\pm \lvert h_Q^\eta(x) \rvert$ and that $\lvert h_Q^\eta \rvert$ is constant on $Q$. Hence
\begin{equation*}
\begin{split}
&\be \norm{\sum_{Q\subseteq S} \epsilon_Q h^\eta_Q \langle \langle b, h^\eta_Q \rangle f \rangle_Q 1_Q}_{L^p(E)}^p
= \be \norm{\sum_{Q\subseteq S} \epsilon_Q  \langle \langle b, h^\eta_Q \rangle \lvert h^\eta_Q \rvert f \rangle_Q 1_Q}_{L^p(E)}^p.
\end{split}
\end{equation*}
By the vector-valued Stein inequality, and the observation that, at each point $x\in \br^d$, we have $h_Q^\eta(x)=\pm \lvert h_Q^\eta(x) \rvert$, we obtain
\begin{equation*}
\begin{split}
\be \norm{\sum_{Q\subseteq S} \epsilon_Q  \langle \langle b, h^\eta_Q \rangle \lvert h^\eta_Q \rvert  f \rangle_Q 1_Q}_{L^p(E)}^p&\leq \beta_p(E)^p \norm{\sum_{Q\subseteq S} \epsilon_Q \langle b, h^\eta_Q \rangle \lvert h^\eta_Q \rvert f }_{L^p(E)}^p\\
&=\beta_p(E)^p  \norm{\sum_{Q\subseteq S} \epsilon_Q \langle b, h^\eta_Q \rangle h^\eta_Q f }_{L^p(E)}^p
\end{split}
\end{equation*}
By assumption, we have $b:\br^d \to \ct$ with $\ct\subseteq \cl(E)$. By the pointwise norm estimate,
$$
\norm{\sum_{Q\subseteq S} \epsilon_Q \langle b, h^\eta_Q \rangle  h^\eta_Q f }_{L^p(E)}\leq \norm{f}_{L^\infty(E)} \norm{\sum_{Q\subseteq S} \epsilon_Q \langle b, h^\eta_Q \rangle  h^\eta_Q }_{L^p(\ct)}.
$$

We can view $\langle b, h^\eta_Q \rangle  h^\eta_Q$ as a subsequence of a martingale difference sequence (thanks to Emil Vuorinen for pointing this out!). We split $Q$ into two subsets $Q^\eta_+$ and $Q^\eta_-$ according to the value of $h^\eta_Q$,$$Q^\eta_+:=\bigcup_{\substack{Q'\in\ch(Q) : \\ \langle h_Q^\eta\rangle_{Q'}=+\lvert Q \rvert^{-1/2}}}Q' \quad\text{ and }\quad Q^\eta_-:=\bigcup_{\substack{Q'\in\ch(Q) : \\ \langle h_Q^\eta\rangle_{Q'}=-\lvert Q \rvert^{-1/2}}}Q'.$$ The corresponding martingale differences are
$$
U^\eta_Q b:=-\langle b \rangle_Q 1_Q+(\langle b \rangle_{Q_-^\eta} 1_{Q_-^\eta}+\langle b \rangle_{Q^\eta_+} 1_{Q^\eta_+}),$$
and
$$
V^\eta_Q b:=-(\langle b \rangle_{Q_-^\eta} 1_{Q_-^\eta}+\langle b \rangle_{Q^\eta_+} 1_{Q^\eta_+})+\sum_{Q'\in\ch(Q)}\langle f \rangle_{Q'}1_{Q'}.
$$
By construction, $D_Qb=U_Q^\eta b+V_Q^\eta b$ and $U_Q^\eta b=\langle b, h^\eta_Q \rangle h^\eta_Q$.  Hence, for any signs $\epsilon_Q$, we have
\begin{equation*}
\begin{split}
&\norm{\sum_{Q\subseteq S} \epsilon_Q \langle b, h^\eta_Q \rangle  h^\eta_Q }_{L^p(\ct)}=\norm{\sum_{Q\subseteq S} (\epsilon_Q U^\eta_Q b+0\cdot V^\eta_Q b) }_{L^p(\ct)}\\
& \leq \beta_p(\ct) \norm{\sum_{Q\subseteq S} (U^\eta_Q b+ V^\eta_Q b) }_{L^p(\ct)}=\beta_p(\ct) \norm{\sum_{Q\subseteq S} D_Qb}_{L^p(\ct)}.
\end{split}
\end{equation*}
We can expand $\sum_{Q\subseteq S} D_Qb=1_S(b-\langle b \rangle_S)$. By the definition of the BMO space,
$$
\norm{1_S(b-\langle f \rangle_S)}_{L^p(\ct)}\leq \norm{b}_{\text{BMO}_p(\ct)} \lvert S \rvert^{1/p}.
$$
\end{proof}
\begin{proof}[Proof with the decoupling equality]

By the decoupling equality, Theorem \ref{martingaledecoupling}, 
\begin{equation*}
\begin{split}
&\norm{ \sum_{Q\in\cq(S)} D_Qb \langle f \rangle_Q}^p_{L^p(\dx\dmu(y))}\\
&\leq \beta_p(E)^p \be \norm{ \sum_{Q\in\cq(S)} \epsilon_Q D_Qb(y_Q)1_Q(x) \langle f \rangle_Q}^p_{L^p(\dx\dmu(y))}.
\end{split}
\end{equation*}
Now, at each point $y_Q\in Q$, we have $ D_Qb(y_Q)\langle f \rangle_Q=\langle  D_Qb(y_Q) f \rangle_Q$. By the vector-valued Stein inequality, 
\begin{equation*}
\begin{split}
&\be \norm{ \sum_{\substack{Q\in\cd :\\ Q \subseteq S}} \epsilon_Q 1_Q(x) \langle D_Qb(y_Q)f \rangle_Q}^p_{L^p(\dx;E)}\\
&\leq \beta_p(E)^p\be \norm{ \sum_{\substack{Q\in\cd :\\ Q \subseteq S}} \epsilon_Q 1_Q(x) D_Qb(y_Q)f(x) }^p_{L^p(\dx \dmu(y);E)}.
\end{split}
\end{equation*}
By the pointwise norm estimate, 
\begin{equation*}
\begin{split}
&\be \norm{ \sum_{\substack{Q\in\cd :\\ Q \subseteq S}} \epsilon_Q 1_Q(x) D_Qb(y_Q)f(x) }^p_{L^p(\dx\dmu(y);E)}\\
&\leq \norm{f}_{L^\infty(E)}^p \be \norm{\sum_{\substack{Q\in\cd :\\ Q \subseteq S}} \epsilon_Q 1_Q(x) D_Qb(y_Q)}^p_{L^p(\dx \dmu(y);\ct)}.
\end{split}
\end{equation*}
By the decoupling equality, Theorem \ref{martingaledecoupling},
$$
\norm{\sum_{\substack{Q\in\cd :\\ Q \subseteq S}} \epsilon_Q 1_Q(x) D_Qb(y_Q)}^p_{L^p(\dx \dmu(y);\ct)}\leq \beta_p(\ct)^p \norm{ \sum_{\substack{Q\in\cd :\\ Q \subseteq S}} D_Qb(x) }^p_{L^p(\dx;\ct)}.
$$
\end{proof}
\begin{remark}In the scalar-valued setting, we obtain the following proof of the boundedness of the dyadic paraproduct: Let $\cs$ be the collection of dyadic cubes that is iteratively chosen by the condition $\langle \abs{f} \rangle_{S'}> 2 \langle \abs{f} \rangle_S$. Hence $\abs{\langle f \rangle_Q}\leq 2 \langle \abs{f} \rangle_S$ whenever $\pis(Q)=S$. From the variant of Pythagoras' theorem (Lemma \ref{pythagorassparse}), Burkholder's inequality, and the special case of the dyadic Carleson embedding theorem (Lemma \ref{carlesonembedding}), it follows that
\begin{equation*}
\begin{split}
\norm{\sum_Q  \langle f \rangle_Q D_Qb}_{L^p(\br^d;\br)}&\leq 3 p \Big( \sum_{S} \norm{\sum_{\pis(Q)=S} \langle f \rangle_Q D_Qb}_{L^p(\br^d;\br)}^p\Big)^{1/p}\\
&\leq 3p 2\beta_p(\br) \Big(\sum_{S} \langle \abs{f}\rangle_S ^p \norm{\sum_{Q: Q\subseteq S} D_Qb}_{L^p(\br^d;\br)}^p\Big)^{1/p}\\
&\leq  6p \beta_p(\br)  \norm{b}_{\text{BMO}_p(\br^d;\br)}\Big(\sum_{S} \langle \lvert f \rvert \rangle ^p \lvert S \rvert \Big)^{1/p}\\
&\leq 6p \beta_p(\br) \norm{b}_{\text{BMO}_p(\br^d;\br)} 2 p' \norm{f}_{L^p(\br^d;\br)}.
\end{split}
\end{equation*}
Note that $\beta_p(\br)=\max\{p,p'\}-1$, which was proven by Burkholder \cite{burkholder1984d}.
\end{remark}
\section{Vector-valued dyadic representation theorem}\label{section_representation}
The proof of the vector-valued dyadic representation theorem follows verbatim the proof of the scalar-valued one that is given in Hyt\"onen's lecture notes on the $A_2$ theorem \cite{hytonen2011}, except for the estimation of matrix elements: In the scalar-valued case, the absolute value of the matrix elements (which are real numbers) is estimated, whereas in the vector-valued case, the $R$-bound of the matrix elements (which are operators) needs to be estimated. For readability, we have sketched the whole proof here. 
\subsection{Expanding the dual pairing by means of dyadic shifts}\label{scalarproof}
By expanding $g\in\lpdual$ as $$g=\sum_{J\in\cd} D_J g=\sum_{J\in\mathcal{D}}\sum_{\eta=1}^{2^d-1}\langle g, h^\eta_J \rangle h_J^\eta,$$ where $ h_J^\eta$ with $\eta=1,\ldots,2^d-1$ and $J\in\cd$ are $L^2$-normalized Haar functions, and $f\in\lpdirect$ similarly, the dual pairing is written as
$$
\langle g, Tf \rangle = \sum_{\substack{I\in\mathcal{D},J\in\mathcal{D}}} \langle g,h_J \rangle \langle h_J,T h_I \rangle \langle h_I, f\rangle.
$$  
The index $\eta$ will be suppressed from now on.
To control the relative arrangement of $I$ and $J$ and whence the size of matrix elements, the notion of a good dyadic cube is introduced.
\begin{definition}[Good dyadic cube]\label{def_goodcube}
Fix a {\it boundary exponent} $\gamma\in (0,1)$ and an {\it ancestor threshold} $r\in\mathbb{N}$. A dyadic cube $I\in\mathcal{D}$ is {\it good} if we have
$$
\mathrm{dist}(I,K^c)>\left(\frac{\ell(I)}{\ell(K)}\right)^{\gamma}\ell(K)
$$
for every dyadic ancestor $K\in\mathcal{D}$ of the dyadic cube $I$ such that $\ell(K)\geq 2^r\ell(I)$. 
\end{definition}

To restrict to the good cubes in the dual pairing, the randomized dyadic systems are introduced.
Let $\mathcal{D}^0$ designate the standard dyadic system. For every parameter $(\omega_j)_{j\in\mathbb{Z}}\in (\{0,1\}^d)^{\mathbb{Z}}=:\Omega$ and every $I\in\mathcal{D}^0$, the translated dyadic cube $I\dot+ \omega$ is defined by $$
I\dot+ \omega:= I+\sum_{j: 2^{-j} < \ell(I)}2^{-j}\omega_j.
$$
For each $\omega\in\Omega$, the translated dyadic system $\mathcal{D}^\omega$ is defined by $\mathcal{D}^\omega:=\{I\dot+\omega : I\in\mathcal{D}^0\}. $ The parameter set is equipped with the natural probability measure: Each component $\omega_j\in \{0,1\}^d$ has an equal probability $2^{-d}$ of taking any of the $2^d$ values and all components are stochastically independent. By construction, the position and the goodness of a dyadic cube $I\dot+\omega$ are stochastically independent. Also by construction, the probability 
$
P_\omega(\{\text{$I\dot+\omega\in \mathcal{D}^\omega$ is good}\})=:\pi_\text{good}
$
does not depend on $I \in \mathcal{D}^0$, and, as calculated in \cite[Lemma 2.3]{hytonen2011},
\begin{equation*}\label{goodprobability}
\pi_\text{good}\geq 1- \frac{8d}{\gamma} 2^{-r\gamma}.
\end{equation*}
In particular, for any boundary exponent $\gamma\in(0,1)$ we can make the probability $\pi_\text{good}$ strictly positive by choosing the ancestor threshold $r\in\bn$ sufficiently large.

The following proposition was proven by Hyt\"onen \cite[Proposition 3.5]{hytonen2011}. (For an earlier version of the proposition, see \cite[Theorem 3.1]{hytonen2012a}.)

\begin{proposition}[Discarding the bad cubes]\label{scalaraveratingtrick}
Assume that $T:L^p(\mathbb{R}^d;\mathbb{R})\to L^p(\mathbb{R}^d;\mathbb{R})$ is bounded. Then
$$
\langle \beta, T\alpha \rangle =\frac{1}{\pi_\text{good}}\mathbb{E}_\omega \sum_{\substack{I\in\mathcal{D}^\omega,J\in\mathcal{D}^\omega:\\ \mathrm{smaller}\{I,J\}\text{ is good } }} \langle \beta,h_J \rangle \langle h_J,T h_I \rangle \langle h_I, \alpha\rangle
$$  
for all $\beta\in C^1_0(\mathbb{R}^d;\mathbb{R})$ and $\alpha\in C^1_0(\mathbb{R}^d;\mathbb{R})$. 

\end{proposition}

\noindent Let $C^1_0(\mathbb{R}^d;\mathbb{R})\otimes E$ denote the set of all finite linear combinations of the form 
$$
f=\sum_{i=1}^I \alpha_i e_i\text{ with } \alpha_i \in C^1_0(\mathbb{R}^d;\mathbb{R}) \text{ and } e_i\in E,
$$
set which is dense in $L^p(\mathbb{R}^d;E)$.
By linearity, Theorem \ref{scalaraveratingtrick} extends to vector-valued functions.
\begin{corollary}Let $E$ be a Banach space.
Assume that $T:L^p(\mathbb{R}^d;E)\to L^p(\mathbb{R}^d;E)$ is bounded. Then
$$
\langle g, Tf \rangle =\frac{1}{\pi_\text{good}}\mathbb{E}_\omega \sum_{\substack{I\in\mathcal{D}^\omega,J\in\mathcal{D}^\omega:\\ \mathrm{smaller}\{I,J\}\text{is good} }} \langle g,h_J \rangle \langle h_J,T h_I \rangle \langle h_I, f\rangle
$$  
for all $g\in C^1_0(\mathbb{R}^d;\mathbb{R})\otimes E^*$ and $f\in C^1_0(\mathbb{R}^d;\mathbb{R})\otimes E$.
\end{corollary}

Next, the paraproducts are extracted. The dyadic system $\mathcal{D}^\omega$ is suppressed in the notation from now on. Consider the summation
$$
\sum_{\substack{I,J:\\ \mathrm{smaller}\{I,J\}\text{ is good} }} \langle g,h_J \rangle \langle h_J,T h_I \rangle \langle h_I, f\rangle.
$$
In the case \lq$I\subsetneq J$\rq, the paraproduct $\Pi_{T^*1}^*$ is extracted as follows: Let $J_I$ denote the dyadic child of $J$ that contains $I$. Then
\begin{equation*}
\begin{split}
\langle h_J,Th_I \rangle &= \langle 1_{J_I^c}h_J,Th_I \rangle +\langle h_J \rangle_I \langle 1_{J_I} ,Th_I \rangle\\&=\langle 1_{J_I^c}(h_J-\langle h_I \rangle _I),Th_I \rangle +\langle h_J \rangle_I \langle 1_{J_I}+1_{J_I^c} ,Th_I \rangle.
\end{split}
\end{equation*}
Summing the last term yields
\begin{equation}\label{t1interpretation}
\begin{split}
\sum_{I,J: I\subsetneq J} \langle g, h_J \rangle \langle h_J \rangle_I \langle 1 ,Th_I \rangle \langle h_I,f \rangle&=\sum_{I}  \langle \sum_{J : J\supsetneq I}\langle g, h_J \rangle h_J \rangle_I \langle 1 ,Th_I \rangle \langle f,h_I \rangle\\
&=\langle \sum_{I} \langle g \rangle_I \langle 1,Th_I \rangle h_I ,f \rangle=:\langle \Pi_{T^*1}g ,f \rangle.
\end{split}
\end{equation}
Similarly, in the case \lq \lq$J\subsetneq I$\rq\, the paraproduct $\Pi_{T1}$ is extracted. For the remaining, it is supposed that the paraproducts are extracted,
and hence the convention
$$
\langle h_J,Th_I \rangle:=\langle 1_{J_I^c}(h_J-\langle h_I \rangle _I),Th_I \rangle\quad\text{whenever $I\subsetneq J$},
$$
is used together with the similar convention whenever $J\subsetneq I$.

Next, the summation is rearranged according to the minimal common dyadic ancestor of $I$ and $J$, which is denoted by $I\vee J$. (If $I\subseteq J$, then $I\vee J= J$. If $I\cap J=\emptyset$, then a common dyadic ancestor exists because one of the cubes is good.)

By splitting the summation according to which one of the cubes $I$ and $J$ has smaller side length (and hence is good), and  by rearranging the summation according to which cube $K$ is the minimal common dyadic ancestor $I\vee J$ and what is the size of $I$ and $J$ relative to $I\vee J$, one obtains
\begin{equation*}
\begin{split}
\sum_{\substack{I,J:\\ \mathrm{smaller}\{I,J\}\text{ is good}}}&= 
\sum_{i,j: i\geq j}\sum_{K}\sum_{\substack{I,J: I \vee J=K, \\I \text{ is good},\\ \ell(I)=2^{-i}\ell(K),\\\ell(J)=2^{-j}\ell(K)}}   +\sum_{i,j: j> i} \sum_{K}\sum_{\substack{I,J:I \vee J=K,\\ J \text{ is good},\\ \ell(I)=2^{-i}\ell(K),\\\ell(J)=2^{-j}\ell(K)}} .
\end{split}
\end{equation*}
Note that, for $K=I\vee J$, one can write
\begin{equation*}
\begin{split}
&\sum_{\substack{I,J: I \vee J=K,\\I \text{ is good},\\\ell(I)=2^{-i}\ell(K),\\\ell(J)=2^{-j}\ell(K)}}  \langle g, h_J \rangle \langle h_J, T h_I \rangle \langle h_I, f\rangle= \langle g,  D_K^{j}A^{ij}_KD_K^i f\rangle
\end{split}
\end{equation*}
by defining 
$$A^{ij}_Kf(x')=\frac{1_K(x')}{\lvert K \rvert}\int_Ka^{ij}_K(x',x)f(x)\,\mathrm{d}x$$
with
$$
a^{ij}_K(x',x):=\lvert K \rvert \sum_{\substack{I,J:I \vee J=K,\\ \mathrm{smaller}\{I,J\} \text{ is good},\\\ell(I)=2^{-i}\ell(K),\\\ell(J)=2^{-j}\ell(K)}}  h_J(x')h_I(x) \langle h_J, Th_I \rangle.
$$
Altogether, it is obtained that
$$
\langle g, Tf \rangle =\frac{1}{\pi_\text{good}}\mathbb{E}_\omega \sum_{i,j} \langle g, \sum_{K\in\cd^\omega} D^j_K A^{ij}_K D^i_K f \rangle+ \frac{1}{\pi_\text{good}}\mathbb{E}_\omega \langle g, \big(\Pi_{T1}^{\cd^\omega}+(\Pi_{T^*1}^{\cd^\omega})^*\big)f \rangle.
$$
\subsection{Estimating the $R$-bounds of the matrix elements}We may consider the case $i\geq j$ (which means $\ell(I)\leq \ell(J)$), since, by duality, the case $i>j$ can be treated similarly. It remains to estimate the $R$-bound of the family $\{a^{ij}_K(x,x') : K\in\cd, x\in K,x'\in K\}$ of the operator-valued kernels defined by
$$
a^{ij}_K(x',x):=\lvert K \rvert \sum_{\substack{I,J: I \vee J=K \\ I \text{ is good},\\\ell(I)=2^{-i}\ell(K),\\\ell(J)=2^{-j}\ell(K)}}  h_J(x')h_I(x) \langle h_J, Th_I \rangle
$$
with $i\geq j$ (and hence $\ell(I)\leq \ell(J)$).
 We divide this into cases according to two criteria. The first criterion is whether $K$ is much bigger than $I$. The second criterion is how the cubes $I$ and $J$ intersect: Whether $I\subsetneq J$ (in which case $K=J$), $I=J$ (in which case $K=I=J$), or $I\cap J=\emptyset$. In total, we have five cases: \begin{itemize}
\item $\ell(K)>2^r\ell(I) \text{ and } I\cap J=\emptyset$,
\item $\ell(J)>2^{r}\ell(I) \text{ and } I\subsetneq J$ (in this case $K=J$),
\item $\ell(K)\leq 2^r\ell(I) \text{ and } I\cap J=\emptyset$,
\item $\ell(J)\leq 2^r\ell(I) \text{ and } I\subsetneq J$ (in this case $K=J$), and
\item $I=J$ (in this case $K=I=J$).
\end{itemize}
These cases are tackled in Lemmas \ref{firstcase} through \ref{lastcase}, which complete the proof of the representation theorem by assuring that
$$
\mathcal{R}(\{a^{ij}_{K}(x',x) : K\in\mathcal{D},x\in K,  x'\in K \})\lesssim_{r,\gamma,d}(\mathcal{R}_{\text{CZ}_0}+\mathcal{R}_{\text{CZ}_\alpha}+\mathcal{R}_{\text{WBP}}) 2^{-(1-\epsilon)\alpha\max\{i,j\}},
$$ 
under the choice $\gamma:=\frac{\epsilon \alpha}{\alpha+d}$ of the boundary exponent $\gamma\in(0,1)$.

\begin{lemma}[Case \lq$\ell(I)\leq\ell(J), \ell(K)>2^r\ell(I) \text{, and } I\cap J=\emptyset$\rq]\label{firstcase}Suppose that $i$ and $j$ are nonnegative integers such that $i>r$ and $i\geq j$. Let
$$
a^{ij}_{K}(x',x):=\lvert K \rvert \sideset{}{'}\sum_{I,J} \langle h_J,Th_I\rangle h_I(x)h_J(x'),
$$ 
where the summation is over all the dyadic cubes $I$ and $J$ such that $I\cap J=\emptyset$, $I\vee J=K$, $\ell(I)=2^{-i}\ell(K)$, $\ell(J)=2^{-j}\ell(K)$, and $I$ is good with threshold $r$ and exponent $\gamma$. Then
$$
\mathcal{R}(\{a^{ij}_{K}(x',x) : K\in\mathcal{D}, x\in K \text{ and } x'\in K \})\lesssim\mathcal{R}_{\text{CZ}_\alpha}2^{-i(\alpha(1-\gamma)-\gamma d)}.
$$
\end{lemma}
\begin{proof}
We observe that for each triplet $(K, x, x')$ either the sum is empty or there is a unique $I_{K,x}$ and a unique $J_{K,x'}$ satisfying the summation condition. Let $y_{I_{K,x}}$ denote the center of the dyadic interval $I_{K,x}$. By using the integral representation of the Calder\'on--Zygmund operator $T$, and by using the cancellation of the Haar functions, we write
\begin{equation*}
\begin{split}
a^{ij}_{K}(x',x)=&\int_{\mathbb{R}^d\times\mathbb{R}^d}(k(y',y)-k(y',y_{I_{K,x}}))\left(\frac{\lvert y-y'\rvert }{\lvert y-y_{I_{K,x}} \rvert}\right)^{\alpha}\lvert y-y' \rvert^{d}1_{I_{K,x}}(y)1_{J_{K,x'}}(y')\\
&\times \lvert K \rvert \left(\frac{\lvert y-y_{I_{K,x}} \rvert}{\lvert y-y'\rvert}\right)^{\alpha}\frac{1}{\lvert y-y' \rvert^{d}}h_{I_{K,x}}(y)h_{J_{K,x'}}(y')h_{I_{K,x}}(x)h_{J_{K,x'}}(x')\,\mathrm{d}y\,\mathrm{d}y'\\
=:& \int_{\mathbb{R}^d\times\mathbb{R}^d}L_{K,x,x'}(y,y')\times \lambda_{K,x,x'}(y,y')\,\mathrm{d}y\,\mathrm{d}y'.
\end{split}
\end{equation*}

\noindent Under the assumptions, we have $\lvert y-y_{I_{K,x}} \rvert< \frac{1}{2} \lvert y-y'\rvert$, which is checked in the following paragraph. Hence, by the Rademacher standard estimates, $$\mathcal{R}(\{L_{K,x,x'}(y,y') : x\in K,x'\in K \text{ and } y\in \mathbb{R}^d, y'\in\mathbb{R^d}\})\leq \mathcal{R}_{\text{CZ}_\alpha}.$$

Next, we show that 
$$
\sup\{ \int_{\mathbb{R}^d\times\mathbb{R}^d} \lvert \lambda_{K,x,x'}(y,y') \rvert\, \mathrm{d}y'\,\mathrm{d}y : x\in K,x'\in K\} \lesssim_\gamma 2^{-i(\alpha(1-\gamma)-\gamma d)},
$$
which, by Theorem \ref{propertyintegral}, completes the proof. For the remaining, we suppress the dependence on the triplet $(K,x,x')$ in the notation. 
%
Since $y\in I$ and $y_I \in I$, we have $\lvert y-y_I \rvert\leq \frac{1}{2}\ell(I)$, and since $y\in I$ and $y'\in J$, we have $\lvert y-y' \rvert\geq \mathrm{dist}(I,J)$; hence $$
 \left(\frac{\lvert y-y_I \rvert}{\lvert y-y'\rvert}\right)^{\alpha}\frac{1}{\lvert y-y' \rvert^{d}}\leq \left(\frac{\ell(I)}{\mathrm{dist}(I,J)}\right)^\alpha \frac{1}{\mathrm{dist}(I,J)^d}. $$ Therefore
\begin{equation*}
\begin{split}
\int_{\mathbb{R}^d\times\mathbb{R}^d} \lvert \lambda(y,y') \rvert\, \mathrm{d}y'\,\mathrm{d}y \leq \lVert h_I \rVert_\infty \lVert h_J \rVert_\infty \lVert h_I \rVert_1 \lVert h_J \rVert_1 \lvert K \rvert \left(\frac{\ell(I)}{\mathrm{dist}(I,J)}\right)^\alpha \frac{1}{\mathrm{dist}(I,J)^{d}}.&
\end{split}
\end{equation*}
It remains to check that $$
\mathrm{dist}(I,J)\geq 2^\gamma \left(\frac{\ell(I)}{\ell(K)}\right)^\gamma \ell(K).
$$
In particular, this implies that $\lvert y-y_I \rvert \leq \frac{1}{2} \lvert y-y' \rvert$. Let $K_I$ denote the dyadic child of $K$ that contains $I$. Since $\ell(K)>2^r\ell(I)$, we have $\ell(K_I)\geq 2^r\ell(I)$. Therefore, since $I$ is good, we have that $$\mathrm{dist}(I,K_I^c)> \left(\frac{\ell(I)}{\ell(K_I)}\right)^\gamma \ell(K_I)=2^{\gamma} \left(\frac{\ell(I)}{\ell(K)}\right)^\gamma \ell(K).$$
If $K_I$ intersected $J$, then either $K_I\subsetneq J$ (which is not true because we assume that $I$ and $J$ are disjoint) or $K_I\supseteq J$ (which is not true because we assume that $K$ is the minimal dyadic ancestor of $I$ that contains $J$). Therefore $K_I$ does not intersect $J$, and hence 
$$
\mathrm{dist}(I,J)> \mathrm{dist}(I,K_I^c).
$$
%
The proof is completed.
\end{proof}

\begin{lemma}[Case \lq$\ell(I)\leq\ell(J), \ell(K)\leq 2^r\ell(I) \text{, and } I\cap J=\emptyset$\rq]Suppose that $i$ and $j$ are nonnegative integers such that $i\leq r$ and $i\geq j$. Let
$$
a^{ij}_{K}(x',x):=\lvert K \rvert \sideset{}{'}\sum_{I,J} \langle h_J,Th_I\rangle h_I(x)h_J(x'),
$$ 
where the summation is over all the dyadic cubes $I$ and $J$ such that $I\cap J=\emptyset$, $I\vee J=K$, $\ell(I)=2^{-i}\ell(K)$, $\ell(J)=2^{-j}\ell(K)$, and $I$ is good with threshold $r$ and exponent $\gamma$. Then
$$
\mathcal{R}(\{a^{ij}_{K}(x',x) : K\in\mathcal{D}, x\in K \text{ and } x'\in K \})\lesssim_{r,d}\mathcal{R}_{\text{CZ}_0}
$$
\end{lemma}
\begin{proof}
We note that for each triplet $(K, x, x')$ either the sum is empty or there is a unique $I_{K,x}$ and a unique $J_{K,x'}$ satisfying the summation condition. By using the integral representation of the Calder\'on-Zygmund operator $T$, we write

\begin{equation*}
\begin{split}
a^{ij}_{K}(x',x)=&\int_{\mathbb{R}^d\times\mathbb{R}^d}1_{I_{K,x}}(y)1_{J_{K,x'}}(y')k(y',y)\lvert y-y' \rvert^{d}\\
&\times \lvert K \rvert \frac{1}{\lvert y-y' \rvert^{d}}h_{I_{K,x}}(y)h_{J_{K,x'}}(y')h_{I_{K,x}}(x)h_{J_{K,x'}}(x')\,\mathrm{d}y\mathrm{d}y'\\
=:& \int_{\mathbb{R}^d\times\mathbb{R}^d}L_{K,x,x'}(y,y')\times \lambda_{K,x,x'}(y,y')\,\mathrm{d}y\mathrm{d}y'.
\end{split}
\end{equation*}
By the Rademacher standard estimates, $$\mathcal{R}(\{L_{K,x,x'}(y,y') : x\in K,x'\in K \text{ and } y\in \mathbb{R}^d, y'\in\mathbb{R^d}\})\leq \mathcal{R}_{\text{CZ}_0}. $$
We next check that 
$$
\sup\{ \int_{\mathbb{R}^d\times\mathbb{R}^d} \lvert \lambda_{K,x,x'}(y,y') \rvert\, \mathrm{d}y'\,\mathrm{d}y : x\in K,x'\in K\} \lesssim_{r,d} 1,
$$
which, by Theorem \ref{propertyintegral}, completes the proof.

For the remaining, we suppress the dependence on the triplet $(K,x,x')$ in the notation. Since $\ell(K)\leq 2^r \ell(I)$ and $K\supseteq I$, we have $2^{r+1}I\supseteq K$, and since $K\supseteq J$ and $I\cap J=\emptyset$, we have $(K\setminus I) \supseteq J$; hence $((2^{r+1}I)\setminus I)\supseteq J$. Since $\ell(I)\leq \ell(J)$, $J\subseteq K$ (and hence $\ell(J)\leq\ell(K)$), and $\ell(K)\leq 2^r\ell(I)$, we have $\lvert I \rvert \eqsim_r \lvert J \rvert \eqsim_r \lvert K \rvert$. 
Therefore
\begin{equation*}
\begin{split}
\int_{\mathbb{R}^d\times\mathbb{R}^d} \lvert  \lambda(y,y') \rvert\mathrm{d}y'\,\mathrm{d}y &\leq \lvert K \rvert \lVert h_I \rVert^2_\infty \lVert h_J \rVert^2_\infty\int_I \int_J \frac{1}{\lvert y-y' \rvert^d}\mathrm{d}y'\,\mathrm{d}y\\
&\leq \lvert K \rvert \lVert h_I \rVert^2_\infty \lVert h_J \rVert^2_\infty\int_I \int_{(2^{r+1}I)\setminus I} \frac{1}{\lvert y-y' \rvert^d}\mathrm{d}y'\,\mathrm{d}y\\
&\lesssim_{r,d} \lvert K \rvert \lVert h_I \rVert^2_\infty \lVert h_J \rVert^2_\infty \lvert I \rvert\eqsim_r 1.
\end{split}
\end{equation*}
\end{proof}

\begin{lemma}[Case \lq $I=J=K$\rq]Let
$$
a^{00}_{I}(x',x):=\lvert I \rvert \langle h_I,Th_I\rangle h_I(x)h_I(x').
$$ 
Then
$$
\mathcal{R}(\{a_{I}(x',x) : I\in\mathcal{D}, x\in I \text{ and } x'\in I \})\lesssim_{d}\mathcal{R}_{\text{CZ}_0}+\mathcal{R}_{\text{WBP}}.
$$
\end{lemma}
\begin{proof}
Let $I_i$ (where $i=1,\ldots,2^d$) denote the dyadic children of $I$. By decomposing $1_{I}=\sum_{I_i}1_{I_i}$, and using the integral representation of the Calder\'on--Zygmund operator $T$, we write

\begin{equation*}
\begin{split}
a^{00}_I(x',x)=&\sum_{I_i,I_j} \lvert I \rvert h_I(x)h_I(x') \langle h_I \rangle_{I_i} \langle h_I \rangle_{I_j} \langle 1_{I_i},T1_{I_j} \rangle\\
=&\sum_{I_i} \lvert I \rvert \lvert I_i \rvert h_I(x)h_I(x')  \langle h_I \rangle_{I_i} \langle h_I \rangle_{I_i}  \frac{\langle 1_{I_i},T1_{I_i} \rangle}{\lvert I_i \rvert}\\
&+\sum_{I_i\neq I_j}   \lvert I \rvert \lvert I_i \rvert h_I(x)h_I(x') \langle h_I \rangle_{I_i} \langle h_I \rangle_{I_j}\int_{\mathbb{R}^d\times\mathbb{R}^d}1_{I_i}(y)1_{I_j}(y')k(y,y')\lvert y-y' \rvert^d \\
&\times \frac{1}{\lvert I_i \rvert}\frac{1}{\lvert y-y' \rvert^d}1_{I_i}(y)1_{I_j}(y')\mathrm{d}y\mathrm{d}y'\\
&= \sum_{I_i}  \pm\frac{\langle 1_{I_i},T1_{I_i} \rangle}{\lvert I_i \rvert}+\sum_{I_i\neq I_j}\pm \int_{\mathbb{R}^d\times\mathbb{R}^d} L_{I_i,I_j}(y,y')\times \lambda_{I_i,I_j}(y,y')\mathrm{d}y\mathrm{d}y'.
\end{split}
\end{equation*}
By the Rademacher standard estimates, we have $$
\mathcal{R}(\{L_{I_i,I_j}(y,y') : I\in\mathcal{D},I_i\neq I_j,y\in I_i,y'\in I_j, \}\leq \mathcal{R}_{\text{CZ}_0}.
$$
Moreover, we have
$$
\sup\{ \int_{\mathbb{R}^d\times\mathbb{R}^d} \lvert \lambda_{I_i,I_j}(y,y') \rvert\, \mathrm{d}y'\,\mathrm{d}y : I\in\mathcal{D}, I_i\neq I_j \} \leq \frac{1}{\lvert I \rvert}\int_{I}\int_{(3I)\setminus I}\frac{1}{\lvert y-y'\rvert^d}\lesssim_d 1.
$$
By the Rademacher weak boundedness property, we have
$$\mathcal{R}(\{ \frac{\langle 1_{I},T1_{I} \rangle}{\lvert I \rvert} : I\in\mathcal{D} \})\leq \mathcal{R}_{\text{WBP}}.$$
The proof is completed by using Theorem \ref{propertyintegral} and Proposition \ref{propertytriangle}.
\end{proof}

\begin{lemma}[Case \lq $\ell(I)<2^{-r}\ell(J), I\subsetneq J$]\label{secondcase}Suppose that $i$ is a nonnegative integer such that $i>r$. Let
$$
a^{i0}_{J}(x',x):=\lvert J \rvert \sideset{}{'}\sum_{I} \langle 1_{J^c_I}(h_J-\langle h_J \rangle_{J_I}),Th_I\rangle h_I(x)h_J(x'),
$$ 
where $J_I$ is the dyadic child of $J$ that contains $I$ and the summation is over all the dyadic cubes $I$ such that $I\subsetneq J$, $\ell(I)=2^{-i}\ell(J)$ and $I$ is good with threshold $r$ and exponent $\gamma$. Then
$$
\mathcal{R}(\{a^{ij}_{J}(x',x) : J\in\mathcal{D}, x\in J \text{ and } x'\in J \})\lesssim_{\gamma}\mathcal{R}_{\text{CZ}_\alpha}2^{-i\alpha(1-\gamma)}.
$$
\end{lemma}
\begin{proof}
We observe that for each triplet $(J, x, x')$ either the sum is empty or there is a unique $I_{J,x}$ satisfying the summation condition. Let $y_{I_{J,x}}$ denote the center of the dyadic interval $I_{J,x}$. By using the integral representation of the Calder\'on--Zygmund operator $T$ and by using the cancellation of the Haar functions, we have
\begin{equation*}
\begin{split}
a^{i0}_{J}(x',x)=&\int_{\mathbb{R}^d\times\mathbb{R}^d}(k(y',y)-k(y',y_{I_{J,x}}))\left(\frac{\lvert y-y'\rvert }{\lvert y-y_{I_{J,x}} \rvert}\right)^{\alpha}\lvert y-y' \rvert^{d}\\
&\times \lvert J \rvert \left(\frac{\lvert y-y_{I_{J,x}} \rvert}{\lvert y-y'\rvert}\right)^{\alpha}\frac{1}{\lvert y-y' \rvert^{d}}h_{I_{J,x}}(y)1_{J^c_I}(y')(h_{J}(y')-\langle h_J \rangle_{J_I})\\
&\cdot h_{I_{K,x}}(x)h_{J}(x')\,\mathrm{d}y\mathrm{d}y'\\
=:& \int_{\mathbb{R}^d\times\mathbb{R}^d}L_{J,x,x'}(y,y')\times \lambda_{J,x,x'}(y,y')\,\mathrm{d}y\mathrm{d}y'.
\end{split}
\end{equation*}
Under the assumptions, we have $\lvert y_{I_{J,x}} -y \rvert \leq \frac{1}{2} \lvert y-y' \rvert$, which is checked in next paragraph. Hence, by the Rademacher standard estimates, we have $$\mathcal{R}(\{L_{J,x,x'}(y,y') : x\in J,x'\in J \text{ and } y\in \mathbb{R}^d, y'\in\mathbb{R^d}\})\leq \mathcal{R}_{\text{CZ}_\alpha}. $$
Next, we show that 
$$
\sup\{ \int_{\mathbb{R}^d\times\mathbb{R}^d} \lvert \lambda_{J,x,x'}(y,y') \rvert\, \mathrm{d}y'\,\mathrm{d}y : x\in J,x'\in J\} \lesssim_\gamma 2^{-i\alpha(1-\gamma)},
$$
which, by Theorem \ref{propertyintegral}, completes the proof.

For the remaining, we suppress the dependence on the triplet $(J,x,x')$ in the notation. Since $\mathrm{dist}(I,J_I^c)>\ell(I)$ (which follows from the facts that $I$ is good and $\ell(J_I)\geq 2^r\ell(I))$, since $y\in I$, and since $y'\in J_I^c$, we have that $\lvert y_{I} -y \rvert \leq\frac{1}{2} \lvert y-y' \rvert$. Hence, by the triangle inequality, $\lvert y-y' \rvert \geq \frac{2}{3} \lvert y'-y_I \rvert$. Therefore
$$
\left(\frac{\lvert y-y_I \rvert}{\lvert y-y'\rvert}\right)^{\alpha}\frac{1}{\lvert y-y' \rvert^{d}}\lesssim_{d,\alpha} \ell(I)^\alpha \frac{1}{\lvert y'-y_I \rvert^{\alpha+d}}.
$$
Therefore
\begin{equation}\label{samecalculation1}
\begin{split}
&\int_{\mathbb{R}^d\times\mathbb{R}^d} \lvert \lambda(y,y')\rvert\, \mathrm{d}y'\,\mathrm{d}y\\
&\lesssim_{d,\alpha} \lvert J \rvert \lVert h_{J}-\langle h_J \rangle_{J_I} \rVert_\infty \lVert h_J \rVert_\infty \lVert h_I \rVert_\infty \lVert h_I \rVert_1 \ell(I)^\alpha \int_{J_I^c}\frac{1}{\lvert y'-y_I \rvert^{\alpha+d}}\mathrm{d}y'\\
&\lesssim \left(\frac{\ell(I)}{\mathrm{dist}(I,J_I^c)}\right)^\alpha=\left(\frac{\ell(J_I)}{\mathrm{dist}(I,J_I^c)}\right)^\alpha  \left(\frac{\ell(I)}{\ell(J_I)}\right)^\alpha .
\end{split}
\end{equation}
Since $I$ is good and $\ell(J_I)\geq 2^r \ell(I)$, we have that
$$
\mathrm{dist}(I,J_I^c)>\ell(J_I) \left(\frac{\ell(I)}{\ell(J_I)}\right)^\gamma,
$$
which concludes the proof.
\end{proof}

\begin{lemma}[Case \lq$ \ell(I)\geq 2^{-r}\ell(J), I\subsetneq J$\rq]\label{lastcase}Suppose that $i$ is a nonnegative integer such that $1\leq i\leq r$. Let
$$
a^{i0}_{J}(x',x):=\lvert J \rvert \sideset{}{'}\sum_{I} \langle 1_{J^c_I}(h_J-\langle h_J \rangle_{J_I}),Th_I\rangle h_I(x)h_J(x'),
$$ 
where $J_I$ is the dyadic child of $J$ that contains $I$ and the summation is over all the dyadic cubes $I$ such that $I\subseteq J$, $\ell(I)=2^{-i}\ell(J)$ and $I$ is good with threshold $r$ and exponent $\gamma$. Then
$$
\mathcal{R}(\{a_{J}(x',x) : J\in\mathcal{D}, x\in J \text{ and } x'\in J \})\lesssim_{d,\alpha}\mathcal{R}_{\text{CZ}_\alpha}+\mathcal{R}_{\text{CZ}_0}
$$
\end{lemma}
\begin{proof}
We observe that for each triplet $(J, x, x')$ either the sum is empty or there is a unique $I_{J,x}$ satisfying the summation condition. Let $y_{I_{J,x}}$ denote the center of the dyadic interval $I_{J,x}$. We split $1_{J_I^c}=1_{J_I^c\cap (3I)}+1_{J_I^c\cap (3I)^c}$. By using the integral representation of the Calder\'on--Zygmund operator $T$, and by using the cancellation of the Haar functions, we have
\begin{equation*}
\begin{split}
&a^{i0}_{J}(x',x)\\
=&\int_{\mathbb{R}^d\times\mathbb{R}^d}(k(y',y)-k(y',y_I))\left(\frac{\lvert y-y_{I_{K,x}} \rvert}{\lvert y-y'\rvert}\right)^{-\alpha}\frac{1}{\lvert y-y' \rvert^{-d}}\times \lvert J \rvert \left(\frac{\lvert y-y_{I_{J,x}} \rvert}{\lvert y-y'\rvert}\right)^{\alpha}\\
&\cdot \frac{1}{\lvert y-y' \rvert^{d}}h_{I_{J,x}}(y)1_{J_I^c\cap (3I)^c}(y')(h_{J}(y')-\langle h_J \rangle_{J_I})h_{I_{K,x}}(x)h_{J}(x')\,\mathrm{d}y\mathrm{d}y'\\
&+ \int_{\mathbb{R}^d\times\mathbb{R}^d}k(y',y)\lvert y-y' \rvert^{d} \\
&\times \lvert J \rvert {\lvert y-y' \rvert^{d}}h_{I_{J,x}}(y)1_{J_I^c\cap (3I)}(y')(h_{J}(y')-\langle h_J \rangle_{J_I}) h_{I_{K,x}}(x)h_{J}(x')\,\mathrm{d}y\mathrm{d}y'\\
=:& \int_{\mathbb{R}^d\times\mathbb{R}^d}L_{J,x,x'}(y,y')\times\lambda_{J,x,x'}(y,y')\,\mathrm{d}y\mathrm{d}y'\\
&+\int_{\mathbb{R}^d\times\mathbb{R}^d}M_{J,x,x'}(y,y')\times\mu_{J,x,x'}(y,y')\,\mathrm{d}y\mathrm{d}y'.
\end{split}
\end{equation*}
By the Rademacher standard estimates, $$\mathcal{R}(\{L_{J,x,x'}(y,y') : x\in J,x'\in J \text{ and } y\in \mathbb{R}^d, y'\in\mathbb{R^d}\})\leq \mathcal{R}_{\text{CZ}_\alpha} $$
and $$\mathcal{R}(\{M_{J,x,x'}(y,y') : x\in J,x'\in J \text{ and } y\in \mathbb{R}^d, y'\in\mathbb{R^d}\})\leq \mathcal{R}_{\text{CZ}_0}.$$
The same calculation as in \eqref{samecalculation1} yields
$$
\sup\{ \int_{\mathbb{R}^d\times\mathbb{R}^d} \lvert \lambda_{J,x,x'}(y,y') \rvert\, \mathrm{d}y'\,\mathrm{d}y : x\in J,x'\in J\} \lesssim_{d,\alpha} 1.
$$
Moreover, we have
\begin{equation*}
\begin{split}
&\sup\{ \int_{\mathbb{R}^d\times\mathbb{R}^d} \lvert \mu_{J,x,x'}(y,y') \rvert\, \mathrm{d}y'\,\mathrm{d}y : x\in J,x'\in J\} \\
&\leq \lvert J \rvert \lVert h_I \rVert_\infty \lVert h_J \rVert_\infty \lVert h_J-\langle h_J \rangle \lVert_\infty \lVert h_I  \rVert_\infty \int_{I}\int_{(3I)\setminus I}\frac{1}{\lvert y-y'\rvert^d} \lesssim_{d} 1.
\end{split}
\end{equation*}
By Theorem \ref{propertyintegral} and Proposition \ref{propertytriangle}, the proof is completed.
\end{proof}

\section*{Acknowledgments}
Both authors are supported by the European Union through the ERC Starting Grant \lq Analytic-probabilistic methods for borderline singular integrals\rq. They are part of the Finnish Centre of Excellence (CoE) in Analysis and Dynamics Research.
\bibliographystyle{plain}
\bibliography{t1operatorkernel}
\end{document}